\documentclass[10.5pt,letterpaper,leqno]{article}
\usepackage[letterpaper, hmargin=2.2cm, vmargin={2cm, 2cm}]{geometry}
\usepackage{myDefs}
\usepackage[symbol]{footmisc}

\begin{document}
\title{Bounds on the growth of subharmonic frequently oscillating functions}
\author{Adi Gl\"ucksam\thanks{Supported in part by Schmidt Futures program}}
\maketitle

\begin{abstract}
We present a Phragm\'en-Lindel\"of type theorem with a flavour of Nevanlinna's theorem for subharmonic functions with frequent oscillations between zero and one. We use a technique inspired by a paper of Jones and Makarov.
\end{abstract}
\section{Introduction}
In this note we introduce a Phragm\'en-Lindel\"of type theorem for subharmonic functions which are bounded on a set, possessing a special geometric structure. The bounds we obtain depend heavily on the structure of the set, and are proven to be optimal. A different but valid angle on the matter, would be to say we present a more elaborate version of Nevanlinna's theorem for a special class of subharmonic functions with a `well distributed' zero set.

To formally state our results, we will need the following definitions:
\subsection{Definitions}
A cube $I$ is called a {\it basic cube} if $I=\prodit j 1 d [n_j,n_j+1)$ for $n_1,\cdots,n_d\in\Z$, i.e $I$ is a half open half closed cube with edge length one, whose vertices lie on the lattice $\Z^d$. Given a subharmonic function $u$ and a basic cube $I$, we consider the properties: 
\begin{center}
\begin{minipage}[b]{.3\textwidth}
\vspace{-\baselineskip}
\begin{equation*}
\tag{P1}\underset{x\in I}\sup\; u(x)\ge 1 \label{eq:max}
\end{equation*}
\end{minipage}%
\hspace{0.2\textwidth}
\begin{minipage}[b]{.4\textwidth}
\vspace{-\baselineskip}
\begin{equation*}
\tag{P2} \lambda_{d-1}(I\cap Z_u)\ge\eps_d \label{eq:zero}
\end{equation*}
\end{minipage}
\end{center}
for some constant $\eps_d>0$, where $Z_u:= \bset{u\le 0}$ and $\lambda_{d-1}$ denotes the $(d-1)$-dimensional Hausdorff content. If a basic cube $I$ satisfies both properties, we say that {\it the function $u$ oscillates in $I$}. Otherwise, we say $I$ is a {\it rogue basic cube}.

Do oscillations have any effect on the growth of a subharmonic function? The following observation suggests that the answer is yes:
\begin{obs}\label{obs:jump}
There exists a constant $\alpha_d$, which depends on the dimension alone, so that for every subharmonic function $u$ defined in a neighbourhood of the unit ball in $\R^d$, $B(0,1)$, if $\lambda_{d-1}\bb{Z_u\cap B\bb{0,\frac12}}>\eps>0$ then
$$
\underset{y\in B(0,1)}\sup\; u(y)\ge u(0)e^{\alpha_d\cdot \eps}.
$$
\end{obs}
To prove this observation we rely on the following claim:
\begin{claim}\label{clm:HarmonicBnd}
Let $E\subset B\bb{0,\frac12}$ be a compact set with $\lambda_{d-1}(E)>0$. Then
$$
\omega(0,E;B(0,1)\setminus E)\gtrsim_d\;\lambda_{d-1}(E).\footnote[1]{Here, and anywhere else in this paper, we write $A\lesssim B$ if there exists $\alpha>0$ for which $A\le\alpha\cdot B$. We write $A\sim B$ if $A\lesssim B$ and $B\lesssim A$. If the constant depends on the dimension $d$, we will add a subscript $d$ to each notation, i.e $A\lesssim_d B$ and $A\sim_d B$.}
$$
\end{claim}

This is definitely known and used by experts in the field. In two words, it is true since the harmonic measure is bounded from bellow by a constant over the energy of the equilibrium measure of the set $E$, which in turn is bounded from bellow by a constant times the $(d-1)$-Hausdorff content of the set $E$. For the reader's convenience, we include a proof of this claim in Appendix A.
\begin{proof}[Proof of Observation \ref{obs:jump}]
Let $E\subseteq Z_u\cap B\bb{0,\frac12}$ be a compact set with $\lambda_{d-1}(E)>\frac12\lambda_{d-1}\bb{Z_u\cap B\bb{0,\frac12}}$. Then, following Claim \ref{clm:HarmonicBnd}, we know that $\omega(0,E;B(0,1)\setminus E)\ge \alpha_d\cdot\eps$ for some uniform constant $\alpha_d$, which depends on the dimension alone. By definition of harmonic measures, and since $u|_E\le 0$
\begin{eqnarray*}
&&u(0)\le\integrate{\partial\bb{B(0,1)\setminus E}}{}{u(y)}\omega(0,y;B(0,1)\setminus E)\le \underset{y\in B(0,1)}\sup\; u(y)\cdot\omega(0,\partial B(0,1)\setminus E;B(0,1)\setminus E)\\
&=&\underset{y\in B(0,1)}\sup\; u(y)\bb{1-\omega(0,E;B(0,1)\setminus E)}\le \underset{y\in B(0,1)}\sup\; u(y)\bb{1-\alpha_d\cdot\eps}\le \underset{y\in B(0,1)}\sup\; u(y)e^{-\alpha_d\cdot\eps},
\end{eqnarray*}
concluding the proof.
\end{proof}
It seems like every oscillation imposes an increment on the function. It is therefore interesting to investigate what is the relationship between the number of basic cubes where the function $u$ oscillates and its minimal possible growth.\\
To describe the growth-rate of a non-negative subharmonic function $u$, we will abuse the notation $M_u$ to denote the function from subsets of $\R^d$ to $\R_+$ defined by $M_u(A):=\underset{z\in A}\sup\; u(z)$ , as well as the function from $\R_+$ to itself defined by $M_u(R):=M_u(B(0,R))$.

In this paper we produce a Phragm\'en–Lindel\"of type theorem to describe bounds on the minimal possible growth of subharmonic functions with frequent oscillations between zero and one. We show that every subharmonic function $u$ defined in a neighbourhood of $[-N,N]^d$, and oscillating in all but $a_N$ basic cubes, has a lower bound
$$
M_u\bb{[-N,N]^d}\ge\exp\bb{c_d\frac{N}{1+\bb{\frac{a_N}{N}}^{\frac1{d-1}}}\log^{\frac d{d-1}}\bb{2+\frac{a_N}{N}}},
$$
as long as $N$ is large enough, and $a_N\le c_0\cdot N^d$. The constants $c_d$ and $c_0$ depend on the dimension $d$. We also show that this bound is optimal.

To describe the asymptotic behavior of the number of rogue basic cubes, those who do not satisfy either property (\ref{eq:max}) or property (\ref{eq:zero}), we use a comparison function:\\
Given a monotone non-decreasing function $f\,(t)\le t^d$, a subharmonic function $u$, and a cube $Q\subset\R^d$ with edge length $\ell(Q)$, we let
$$
\gamma_f^u(Q)=\gamma(Q)=\frac{\#\bset{I\subset Q ,I\text{ is a rogue basic cube}}}{f(\ell(Q))}.
$$
We say $u$ is {\it $f$-oscillating in $Q$} if $u$ is defined in a neighbourhood of $Q$, and $\gamma(Q)<1$. We say $u$ is {\it $f$-oscillating} if
$$
\limitsup N\infty\gamma([-N,N]^d)<1.
$$
Lastly, we remind the reader the definition of regularly varying functions. A function $f$ is called {\it regularly varying of index $a\ge0$} if $f(t)=t^a\cdot g(t)$ and the function $g$ satisfies that for every $\lambda>0$
$$
\limit t\infty \frac{g(\lambda\cdot t)}{g(t)}=1.
$$
The function $g$ is called {\it slowly varying}. For more information, see, for example, page 18 in \cite{Bingang}.
\subsection{History and Motivation}
Let us begin with a simple case: What can be said about the minimal possible growth of a non-negative subharmonic function which oscillates in every basic cube, i.e an $f$-oscillating function for $f(t)=1$?\\
In light of Observation \ref{obs:jump}, a lower bound is not hard to obtain by noting that every ball of radius $2\sqrt d$ includes a basic cube as a subset. We may therefore apply Observation \ref{obs:jump} with 
$\eps=\eps(d)$ to find a sequence of points, $\bset{x_n}\in\R^d$, satisfying that $\abs{x_n}-\abs{x_{n-1}}= 2\sqrt d$ while $u(x_n)\ge e^{\delta_d}u(x_{n-1})\ge\cdots\ge e^{\delta_d\cdot n}$ for some $\delta_d>0$. This implies that
$$
\limitinf r\infty \frac{\log\bb{M_u(B_d(0,r))}}{r}\ge\limitinf n\infty \frac{\log\bb{u(x_n)}}{2n\sqrt d}\ge\frac {\delta_d}{2\sqrt d}>0.
$$
In fact, most of the proofs of lower bounds we will mention, basically count the number of times we can apply Observation \ref{obs:jump} to obtain a lower bound on the growth of the function.\\
To see a construction of an example with this optimal growth, one could extend Theorem 2 in \cite{Us2017} to higher dimensions, which we shall do in Section \ref{sec:upper}.

It is interesting to ask what happens if we relax the condition that $u$ oscillates in {\it every} basic cube. The first result in this direction, which already appears in \cite{Us2017}, was for the case $d=2$ where we allow $u$ to have many rogue basic cubes, their number proportional to the area. Here is a restatement of it:
\begin{lem}{\cite[Lemma 1]{Us2017}}\label{lem:1}
Let $Q$ be a square of edge length $\ell>10$, and let $f(t)=\alpha t^2$ for some $\alpha\in(0,1)$. Then every non-negative subharmonic function $u$ which is $f$-oscillating in $Q$, must satisfy
\[
M_u(Q) \ge e^{c_\alpha\left( \frac{\log \ell}{\log\log \ell} \right)^2}\,.
\]
\end{lem}
In \cite{Us2017}, we also indirectly show a lower bound. We constructed a subharmonic function in $\C$ which is $f$-oscillating for $f(t)=\alpha t^2$ (for some $\alpha\in(0,1)$) while
$$
\limitsup r\infty\frac{\log\bb{M_u(r)}}{\log^{2+\eps}(r)}<\infty.
$$
These functions arise when studying the support of translation invariant probability measures on the space of entire functions, i.e probability measures which are invariant under the action of the group $\C$ acting on the space of entire functions by translations: $(T_wf)(z)=f(z+w)$. For more information, see \cite{Us2017}.\\
The examples we discussed are two extremes. One discusses the case $f(t)=1$, i.e there are no rogue basic cubes. The other deals with the situation where $f(t)=\alpha t^2$ for $\alpha\in(0,1)$, i.e the number of rogue basic cubes in a large cube is proportional to the volume of the cube. It is natural to ask what happens in the intermediate cases, which is the subject of this paper.
\subsection{Results}
Let $f$ be a regularly varying function. What can be said about the minimal possible growth of $f$-oscillating subharmonic functions?
\begin{thm}\label{thm:main}
\begin{enumerate}[label=(\Alph*)]
\item Every $f$-oscillating subharmonic function $u$ satisfies
$$
\limitinf R\infty \frac{\log(M_u(R))}{\frac{R}{1+\bb{\frac{f(R)}{R}}^{\frac1{d-1}}}\log^{\frac d{d-1}}\bb{2+\frac{f(R)}{R}}}>0,
$$
provided that $f(t)\le c_0\cdot t^d$ for all $t$ large, where $c_0$ is a constant, which depends on the dimension alone.
\item Let $f$ be a regularly varying function of index $\alpha\in[0,d]$. Assume that either $\alpha<d$ or $f(t)=t^d\cdot g(t)$ for some monotone non-increasing function $g$. Then there exists an $f$-oscillating subharmonic function $u$ satisfying
$$
\limitsup R\infty \frac{\log(M_u(R))}{\frac{R}{1+\bb{\frac{f(R)}{R}}^{\frac1{d-1}}}\log^{\frac d{d-1}}\bb{2+\frac{f(R)}{R}}}<\infty.
$$
\end{enumerate}
\end{thm}
In particular, if $f(t)=t^\alpha$ then the theorem above shows that the minimal possible growth for $f$-oscillating subharmonic functions is 
$$
\log\bb{M_u(R)}\ge
\begin{cases}
c_d\cdot R &, \alpha\le1\\
c_{d,\alpha}\bb{R^{d-\alpha}\log^d(R)}^{\frac1{d-1}}&, \alpha>1
\end{cases},
$$
and this bound is optimal.

This theorem implies that in Lemma \ref{lem:1}, the $\log\log$ component in the numerator of the exponent's power is not essential. In fact, the following corollary can now be proved-
\begin{cor}
Let $\lambda$ be a non-trivial translation-invariant probability measure on the space of entire functions. Then, for $\lambda$-a.e. function $f$,
$$
\limitinf R\infty\frac{\log\log M_f(R)}{\log^2(R)}>0.
$$
\end{cor}
To prove this corollary, one may use the original proof of Theorem 1 part (A) in \cite{Us2017}, and replace the use of Lemma 1 with part (A) of Theorem \ref{thm:main} above.

Nevertheless, the construction, described in Theorem \ref{thm:main} part (B) above, cannot be used to construct a translation invariant probability measure with a similar upper bound. For more details see the end of Section \ref{sec:upper}.
\subsection{An overview of the paper: Methods and Tools}
\subsubsection{The proof of Theorem \ref{thm:main} part (A)} Given a subharmonic function $u$, there is some connection between bounds on the harmonic measure of its zero set in a large ball, $\omega(0,Z_u,\; B(0,R)\setminus Z_u)$, and estimates on the growth of the function $u$ in $B(0,R)$ (for definitions and basic properties of harmonic measures see for example \cite{Haywood} chapter 3.6).  Though this connection is mostly conceptual, it is not surprising that some of the tools, used in this paper, are also used to estimate harmonic measures.\\
We use a stopping time argument inspired by the one appearing in a paper by Jones and Makarov (see \cite{JoneMakarov1995}).\\
Jones and Makarov's work in \cite{JoneMakarov1995} is very influential in modern geometric function theory. For example, A.Baranov and H.Hedenmalm use some of the techniques in this paper to study integral means spectrum of conformal mappings (see \cite{Baranov2008}).

The proof of Theorem \ref{thm:main} part (A) is based on a main lemma, which essentially produces a lower bound on the quotient $\frac{M_u([-N,N]^d)}{M_u(I)}$ for every $N$ large enough and for at least half of the basic cubes $I\subset[-N,N]^d$.

\begin{figure}[htp]
\begin{minipage}[c]{.55\textwidth}
The main idea of the proof is as follows- for every cube $I$ we look at the growing sequence of cubes concentric with  $I$, and contained in $[-N,N]^d$ (see Figure \ref{fig:layers} to the right). 
We then choose a subsequence of this sequence of cubes, say $\bset{Q_j}$, satisfying that for some $\delta_d>0$, for all $j$, for every $\xi\in \partial Q_j$ there exists $r_\xi$ so that on one hand $B(\xi,r_{\xi})\subset Q_{j+1}$ and on the other hand
$$
\underset j\inf\underset {\xi\in\partial Q_j}\inf\; \frac{\lambda_{d-1}\bb{Z_u\cap B_d(\xi,r_{\xi})}}{r_{\xi}^d}\ge \delta_d.
$$
Using a Observation \ref{obs:jump}, there exists a constant $\alpha_d$ so that if $x_j\in \partial Q_j$ satisfies that $u(x_j)=M_u(Q_j)$, and $r_j=r_{x_j}$ then
\end{minipage}\hfill
\begin{minipage}[c]{.4\textwidth}
\centering
\includegraphics[scale=0.35]{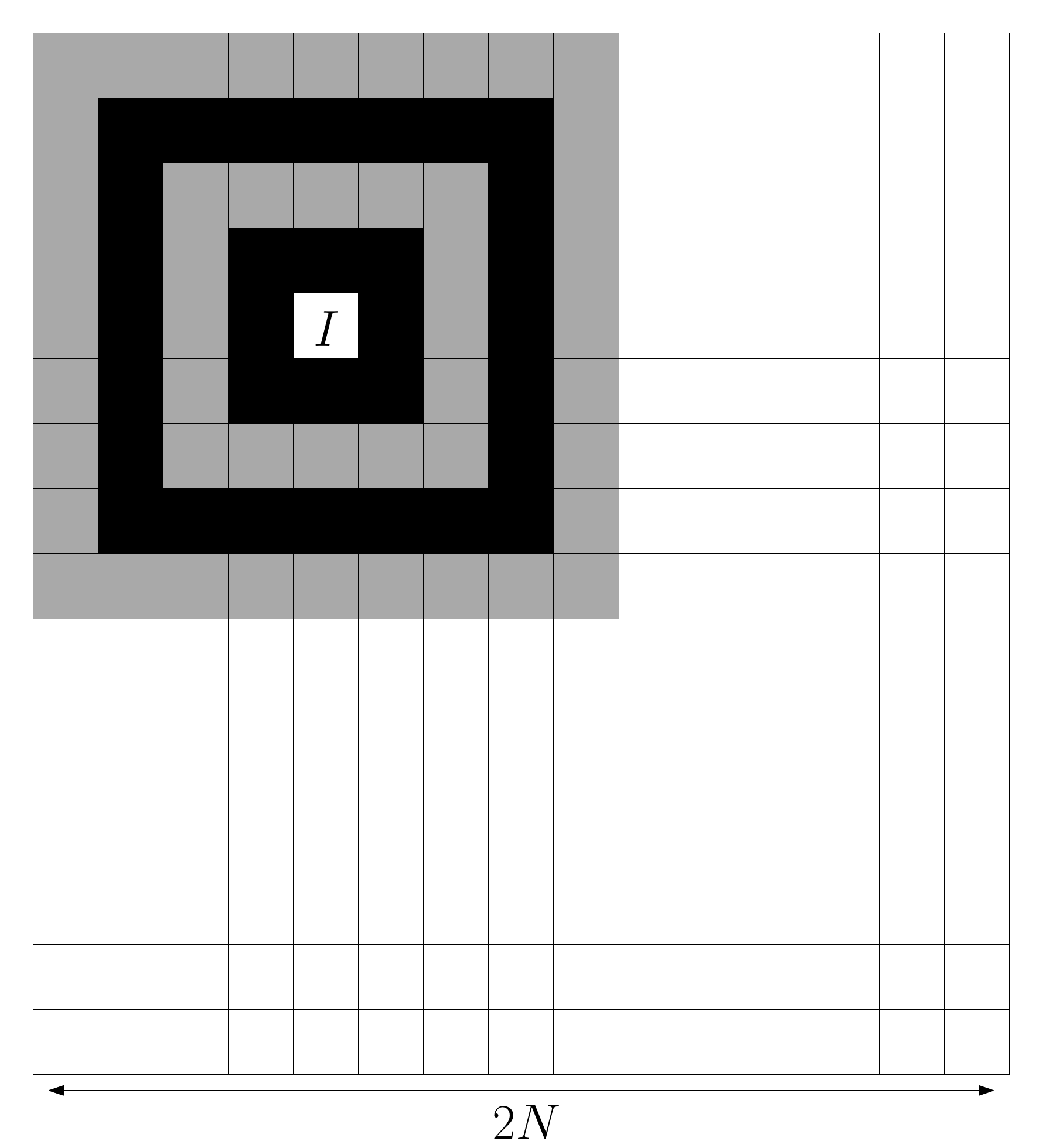}
\caption{In this figure one can see the layers about the cube $I$ which are coloured alternately.}
\label{fig:layers}
\end{minipage}
\end{figure}
\vspace{-4em} 
$$
M_u(Q_j)=u(x_j)\le M_u\bb{B(x_j,r_j)}\cdot e^{-\alpha_d\delta_d}\le M_u\bb{Q_{j+1}}\cdot e^{-\alpha_d\delta_d},
$$
where the last inequality is by inclusion. Doing so recursively we conclude that
$$
M_u(I)\le e^{-\alpha_d\delta_d}M_u(Q_1)\le\cdots\le e^{-\#\bset{Q_j}\alpha_d\delta_d}M_u([-N,N]^d).
$$
It is therefore left the bound from bellow the number of elements in the subsequence $\bset{Q_j}$ to conclude the proof. Here various combinatorial methods are used.\\
It is not the particular statement of this lemma, but the methods used in the proof that are interesting on their own and may find additional applications in many fields. The proof is completely elementary and assumes very basic knowledge of potential theory. The proof of the lower bound, 
Theorem \ref{thm:main} part (A), can be found in Section \ref{sec:lower}, and the formal statement of the lemma and its proof can be found in Subsection \ref{sec:main_lem}.
\subsubsection{The proof of Theorem \ref{thm:main} part (B)} We would have loved to use a construction which is similar to the one presented in \cite{Us2017}. Never the less, such a construction would have two main issues. The first is that the `self similarity' of the function leads to accumulating rogue basic cubes from smaller scales. This accumulation grows like the volume measure and can never work for $f(t)\ll t^d$ (in fact, even for $f(t)=t^d$ we do not get the optimal bound using this method). The second issue is that this method uses hyperplanes to separate the similar copies of the function. This will not work for any function $t\ll f(t)\ll t^{d-1}$ in dimensions higher than two.\\
We will first show that for every $R>0$ there exists a subharmonic function, $u$, such that $u$ oscillates in all but at most $f(2R)$ basic cubes in $[-R,R]^d$. This is only possible on bounded domains, as we use a parameter, which must be positive, while it tends to zero as the diameter of the set tends to infinity. `Glueing' different functions solves the issue of accumulating rogue cubes. To solve the second issue, we will use functions which are non-zero only on tube-like sets, and the `glueing' will be done along tubes as well. The construction can be found in Section \ref{sec:upper}.
\subsubsection{Appendix- another proof of Theorem \ref{thm:main} part (A)} We bring here another proof of Theorem \ref{thm:main} part (A). This proof uses Jones and Makarov's main lemma in \cite{JoneMakarov1995} as a black box, instead of using the ideas presented in it. It was suggested to the author by M. Sodin, and uses potential theory. The proof can be found in Appendix B.
\subsection{Acknowledgements}
The author is lexicographically grateful to Ilia Binder, Persi Diaconis, Eugenia Malinnikova, and Mikhail Sodin for several very helpful discussions. In particular, I would like to extend my gratitude to Chris Bishop who insisted on the connection between the problem presented here and the paper by Jones and Makarov, \cite{JoneMakarov1995}.
\section{Lower bound}\label{sec:lower}
We begin this section with the statement of the main lemma. This lemma is the essence of the proof of the lower bound:
\begin{lem}\label{lem:Jones-Makarov}
Let $Q:=\left[-\frac N2,\frac N2\right]^d$ for some $N\gg1$, and let $u$ be a subharmonic function defined in a neighbourhood of $Q$. Denote by $\mathcal G$ the collection of all basic cubes in $Q$, and let $\mathcal E=\mathcal E_u$ denote sub-collection of cubes that do not satisfy property (\ref{eq:zero}),
$$
\mathcal E=\mathcal E_u:=\bset{I\in \mathcal G; \lambda_{d-1}(I\cap Z_u)<\eps_d}.
$$
There exists a constant $c_0$, depending on the dimension alone, so that if $\#\mathcal E\le c_0 N^d$, then at least half of the basic cubes $I\in\mathcal G$ satisfy:
$$
\frac{M_u(I)}{M_u(Q)}\le\exp\bb{-c_d\frac N{1+\bb{\frac{\#\mathcal E}N}^{\frac1{d-1}}}\log^{\frac d{d-1}}\bb{2+\frac{\#\mathcal E}N}},
$$
where $c_d$ is a constant which depends on the dimension alone as well.
\end{lem}
We will first show how the main lemma implies the proof of Theorem \ref{thm:main} part (A):
\subsection{The proof of Theorem \ref{thm:main} part (A)}
\begin{proof}[\nopunct]
Let $c_0$ be the constant defined in the Main Lemma, Lemma \ref{lem:Jones-Makarov}, let $f$ be a monotone increasing function satisfying that $f(t)\le c_0t^d$ for all large enough $t$, and let $u$ be an $f$-oscillating subharmonic function. Then, using the notation of the lemma, for all $N$ large enough
\begin{eqnarray*}
\#\mathcal E&=&\#\bset{I\in \mathcal G, I\text{ is a basic cube,}\atop\lambda_{d-1}(I\cap Z_u)<\eps_d}=\#\bset{I\in\mathcal G,\; I\text{ is a basic cube which}\atop\text{does NOT satisfy property (\ref{eq:zero})}}\\
&\le&\#\bset{I\in\mathcal G,I\text{ is a rogue basic cube}}=f(N)\cdot\gamma\bb{Q}<f(N)\le c_0N^d.
\end{eqnarray*}
We may therefore apply the Main Lemma to conclude that at least half of the basic cubes in $Q$ satisfy
\begin{eqnarray*}
\frac{M_u(I)}{M_u\bb{Q}}&\le&\exp\bb{-c_d\frac{N}{1+\bb{\frac{\#\mathcal E}N}^{\frac1{d-1}}}\cdot\log^{\frac d{d-1}}\bb{2+\frac{\#\mathcal E}N}}.
\end{eqnarray*}
To get a similar bound  with $f(N)$ instead of $\#\mathcal E$, we note that the function $\psi$ defined by
$$
\psi(x):=\frac1{1+x^{\frac1{d-1}}}\cdot\log^{\frac d{d-1}}\bb{2+x}
$$
is monotone decreasing in some neighbourhood of $\infty$, which depends on the dimension. Combining this with the fact that $\#\mathcal E\le f(N)$, we see that for every $N$ large enough
\begin{eqnarray*}
\frac{M_u(I)}{M_u\bb{Q}}&\le&\exp\bb{-c_d\frac{N}{1+\bb{\frac{\#\mathcal E}N}^{\frac1{d-1}}}\cdot\log^{\frac d{d-1}}\bb{2+\frac{\#\mathcal E}N}}=\exp\bb{-c_d N\cdot \psi\bb{\frac{\#\mathcal E}N}}\\
&\le&\exp\bb{-c_d N\cdot \psi\bb{\frac{f(N)}N}}=\exp\bb{-c_d\frac{N}{1+\bb{\frac{f(N)}N}^{\frac1{d-1}}}\cdot\log^{\frac d{d-1}}\bb{2+\frac{f(N)}N}}.
\end{eqnarray*}
On the other hand, as without loss of generality $c_0<\frac12$, more than half of the basic cubes in $Q$ satisfy property (\ref{eq:max}), that is $M_u(I)\ge 1$. We conclude that there exists at least one cube satisfying Property (\ref{eq:max}) and the inequality above, implying that
\begin{eqnarray*}
M_u\bb{Q}&\ge&\exp\bb{c_d\frac N{1+\bb{\frac{f(N)}N}^{\frac1{d-1}}}\log^{\frac d{d-1}}\bb{2+\frac{f(N)}N}}M_u(I)\\
&\ge&\exp\bb{c_d\frac N{1+\bb{\frac{f(N)}N}^{\frac1{d-1}}}\log^{\frac d{d-1}}\bb{2+\frac{f(N)}N}}.
\end{eqnarray*}
As this holds for every $N\in\N$ large enough,
$$
\limitinf R\infty\frac{\log\bb{M_u(R)}}{\frac{R}{1+\bb{\frac{f(R)}R}^{\frac1{d-1}}}\cdot\log^{\frac d{d-1}}\bb{2+\frac{f(R)}R}}>0
$$
concluding our proof.
\end{proof}
It is left to prove the Main Lemma:
\subsection{The proof of Main Lemma}\label{sec:main_lem}
\vspace{-2em}
\begin{figure}[htp]
\begin{minipage}[b]{.4\textwidth}
Let $K$ denote the approximation from above of the essential part of the set $Z_u\cap Q$ by elements of $\mathcal G\setminus\mathcal E$,
$$
K:=\underset{I\in \mathcal G\atop I\nin\mathcal E}\bigcup I
$$
(see Figure \ref{fig:K} to the right), and define the function 
$\rho(x):= \max\bset{2\sqrt d,r(x)},$ for
\end{minipage}\hfill
\begin{minipage}[b]{.53\textwidth}
\centering
\includegraphics[scale=0.85]{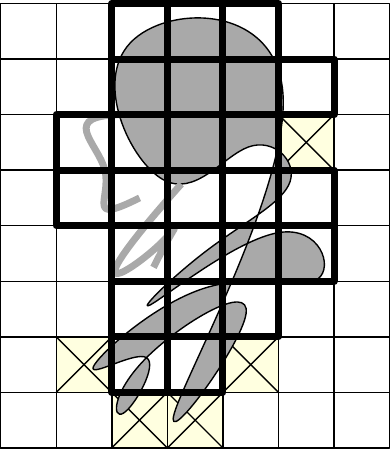}
\caption{The grey area is the set $Z_u$, the cubes with bold boundary are the components of $K$. The cubes marked with X are basic cubes, which intersect $Z_u$ but do not satisfy (\ref{eq:zero}) for $\eps_d=\frac14$.}
\label{fig:K}
\end{minipage}
\end{figure}
\begin{proof}[\unskip\nopunct]
\vspace{-10mm}
$$
r(x):=\inf\bset{t>0,\; m_d\bb{K\cap B\bb{x,\frac t2}}\ge \delta_0\cdot m_d\bb{B(0,1)}\cdot t^d},
$$
where $m_d$ denotes Lebesgue's measure on $\R^d$, and $\delta_0$ is some constant, which depends on the dimension alone.
\begin{obs}\label{obs:Z_u}
There exists a constant $\delta_d\in(0,1)$, which depends on the dimension $d$, so that for every $x$
 $$
 \frac{\lambda_{d-1}\bb{Z_u\cap B(x,\rho(x))}}{\rho(x)^d}\ge \delta_d.
 $$
\end{obs}
\begin{proof}
If $\rho(x)=2\sqrt d$, then there exists $t\in\bb{0,2\sqrt d}$ so that
$$
m_d\bb{K\cap B\bb{x,\frac t2}}\ge\delta_0\cdot m_d\bb{B(0,1)}\cdot t^d,
$$
and in particular $B\bb{x,\frac t2}$ intersects $K$, implying that there exists $I\in \mathcal G\setminus\mathcal E$ so that
$$
I\subset B\bb{x,\frac t2+\sqrt d}\subseteq B\bb{x,2\sqrt d},
$$
since $t\le2\sqrt d\Rightarrow\frac t2+\sqrt d\le 2\sqrt d$. As $I\nin\mathcal E$, it satisfies property (\ref{eq:zero}) and therefore
$$
\lambda_{d-1}\bb{Z_u\cap B\bb{x,2\sqrt d}}\ge \lambda_{d-1}\bb{Z_u\cap B\bb{x,\frac t2+2\sqrt d}}\ge \lambda_{d-1}(Z_u\cap I)\ge\eps_d,
$$
implying that
$$
\frac{\lambda_{d-1}\bb{Z_u\cap B(x,\rho(x))}}{\rho(x)^d}\ge\frac{\eps_d}{2^d d^{\frac d2}}.
$$
If $\rho(x)>2\sqrt d$, then $m_d\bb{K\cap B\bb{x,\frac{\rho(x)}2}}\ge \delta_0\cdot m_d\bb{B(0,1)}\cdot\rho(x)^d$, implying that $B\bb{x,\frac{\rho(x)}2}$ intersects at least $\left\lceil \delta_0\cdot m_d\bb{B(0,1)}\cdot\rho(x)^d\right\rceil$ elements in $\mathcal G\setminus \mathcal E$, by the pigeon-hole principle. On the other hand, $\rho(x)>2\sqrt d$ implying that $\frac{\rho(x)}2+\sqrt d\le\rho(x)$. Combining these two observations together, we conclude that
\begin{eqnarray*}
\lambda_{d-1}\bb{Z_u\cap B\bb{x,\rho(x)}}&\ge& \lambda_{d-1}\bb{Z_u\cap B\bb{x,\frac{\rho(x)}2+\sqrt d}}\ge \eps_d\cdot \#\bset{I\nin \mathcal E, I\cap B\bb{x,\frac{\rho(x)}2}\neq\emptyset}\\
&\ge& \eps_d\cdot\left\lceil\delta_0\cdot m_d\bb{B(0,1)}\cdot\rho(x)^d\right\rceil\ge \eps_d\cdot\delta_0\cdot m_d\bb{B(0,1)}\cdot\rho(x)^d.
\end{eqnarray*}
We conclude the proof by choosing
$$
\delta_d:=\eps_d\cdot\min\bset{2^{-d}\cdot d^{-\frac d2},\delta_0\cdot m_d\bb{B(0,1)}}.
$$
\end{proof}
For every cube $I$ we denote by $A(I,k)$ the $k$th layer of basic cubes surrounding $I$ (see for example the alternating grey and black cubed annuli in Figure \ref{fig:layers} in the introduction). Formally, if $I=\prodit j 1 d \left[n_j,n_j+1\right)$, then
$$
A(I,k)=\prodit j 1 d \left[n_j-k,n_j+1+k\right)\setminus \prodit j 1 d \left[n_j-(k-1),n_j+k\right).
$$
For example in dimension $d=2$ if $I=[0,1)^2$, then $A(I,1)=[-1,2)^2\setminus[0,1)^2$.

We will construct a monotone non-decreasing function  $M:\N\rightarrow\N$, so that for every $k$ fixed it satisfies
\begin{equation*}
\tag{M} \#\bset{I\in\mathcal G, k\in K_I}\ge\frac{11}{12}N^d, \label{eq:M}
\end{equation*}
where $K_I$ is a set defined for every basic cube $I\in\mathcal G$ by
$$
K_I:=\bset{k\in\bset{1,\cdots,\frac {N}{6d}},\; \forall x\in Q\cap A(I,k),\rho(x)\le\ M(k)}.
$$
For now, we assume such a function exists.\\
For every cube $I\in \mathcal G$, define the monotone increasing sequence $\bset{\kappa_j}=\bset{\kappa_j(I)}$ in the following way:
\begin{eqnarray*}
\kappa_1&:=&\min\bset{k\in K_I}\\
\kappa_j&:=&\min\bset{k\in K_I,\; k>\kappa_{j-1}+M(\kappa_{j-1})}.
\end{eqnarray*}
We choose a cube $I$ so that $A\bb{I,\frac {N}{6d}}\subset Q$, and let $\gamma_j$ denote the outer boundary of $A(I,\kappa_j)$, and $Q_j$ denote the closed cube whose boundary is $\gamma_j$. By the way the set $K_I$ was defined, for every $x\in\gamma_j$,
$$
\rho(x)\le M(\kappa_j)\le dist\bb{\gamma_j,\gamma_{j+1}}\Rightarrow B(x,\rho(x))\subset Q_{j+1}.
$$
By using a rescaled version of Observation \ref{obs:jump}, we deduce that, maybe for a different constant $\delta_d$, we have
$$
u(x)\le M_u(B(x,\rho(x)))\cdot e^{-\delta_d}\le M_u(\gamma_{j+1})e^{-\delta_d}= M_u(Q_{j+1})e^{-\delta_d}.
$$
Applying this argument inductively and using the maximum principle, if $Q_0$ denotes the cube $I$, then
$$
\frac{M_u(I)}{M_u(Q)}=\prodit j 0 {\#\bset{\kappa_j}-1} \frac{M_u(Q_j)}{M_u(Q_{j+1})}\le \exp\bb{-\#\bset{\kappa_j}\cdot\delta_d}.
$$
To conclude the proof, we need to bound from bellow $\#\bset{\kappa_j}$ for all but at most half of the basic cubes $I\in \mathcal G$.\\
Note that if the function $M$ is too big, the sequence $\bset{\kappa_j}$ will be a very short sequence, for most cubes. On the other hand, the smaller the function $M$ is, the harder it becomes to satisfy property (\ref{eq:M}).

We shall conclude the proof in three steps:
\vspace{-1em}
\begin{enumerate}[label=Step \arabic*:]
\item
Show that for at least $\frac{10}{11}N^d$ of the basic cubes $I\in\mathcal G$
\begin{equation*}
\tag{$\S$}\#\bset{\kappa_j(I)}\ge \frac1{24}\sumit k 1 {\frac {N}{6d}}\frac 1{M(k)}.\label{eq:Mk_bound}
\end{equation*}
\item Define the function $M$ and prove that it satisfies condition (\ref{eq:M}).
\item Bound the sum $\sumit k 1 {\frac {N}{6d}}\frac 1{M(k)}$ from bellow.
\end{enumerate}
\subsubsection{Step 1:  for at least $\frac{10}{11}N^d$ of the basic cubes $I\in\mathcal G$, (\ref{eq:Mk_bound}) holds.} 
\begin{proof}[\unskip\nopunct]
For every basic cube $I\in\mathcal G$ let
$$
B(I):=\underset{k\in K_I}\sum\frac1{M(k)},
$$
and define the collection
$$
X:=\bset{I,\; B(I)\ge \frac 1{12}\sumit k 1 {\frac {N}{6d}}\frac1{M(k)}}.
$$
We will first show that every basic cube $I\in X$ satisfies (\ref{eq:Mk_bound}). Since $M$ is a monotone non-decreasing function with $M(1)\ge 1$, and by the way the sequence $\bset{\kappa_j}$ was defined, for every $I\in X$
\begin{eqnarray*}
\sumit k 1 {\frac {N}{6d}}\frac 1{M(k)}&\le& 12 B(I)=12\underset{k\in K_I}\sum\frac1{M(k)}=12\sumit j 1 {\#\bset{\kappa_j}} \underset{k\in K_I\atop \kappa_j\le k\le \kappa_j+M(\kappa_j)}\sum\frac 1{M(k)}\\
&\le&12\sumit j 1 {\#\bset{\kappa_j}} \sumit k{\kappa_j}{\kappa_j+M(\kappa_j)}\frac1{M(k)}\le 12\sumit j 1 {\#\bset{\kappa_j}}\frac{M(\kappa_j)+1}{M(\kappa_j)}\le24\#\bset{\kappa_j}.
\end{eqnarray*}
To conclude the proof it is left to show that $\#X\ge\frac{10}{11}N^d$.\\
Following property (\ref{eq:M}),
\begin{equation*}
\tag{$\Sigma$}\underset{I\in \mathcal G}\sum B(I)=\underset{I\in \mathcal G}\sum\underset{k\in K_I}\sum\frac1{M(k)}=\sumit k 1 {\frac {N}{6d}}\frac1{M(k)}\#\bset{I\in \mathcal G,\; k\in K_I}\ge\frac{11}{12}N^d\sumit k 1 {\frac {N}{6d}}\frac1{M(k)}.\label{eq:B(I)}
\end{equation*}
On the other hand, by the way the set $X$ was defined,
$$
B(I)\le \begin{cases}
		\frac 1{12}\sumit k 1 {\frac {N}{6d}}\frac1{M(k)}&, I \nin X\\
		\sumit k 1 {\frac {N}{6d}}\frac1{M(k)}&, I\in X
		\end{cases}.
$$
Then
\begin{eqnarray*}
\underset{I\in \mathcal G}\sum B(I)&=&\underset{I\in X}\sum B(I)+\underset{I\nin X}\sum B(I)\le \#X\sumit k 1{\frac N{6d}}\frac1{M(k)}+\bb{N^d-\#X}\frac1{12}\sumit k 1 {\frac {N}{6d}}\frac1{M(k)}\\
&=&\bb{\frac{11}{12}\#X+\frac1{12} N^d}\sumit k 1 {\frac{N}{6d}}\frac1{M(k)}\overset{\text{by (\ref{eq:B(I)})}}\le\frac{12}{11N^d}\bb{\frac{11}{12}\#X+\frac1{12}N^d} \underset{I\in \mathcal G}\sum B(I)=\bb{\frac{\#X}{N^d}+\frac1{11}}\underset{I\in \mathcal G}\sum B(I)\\
&\Rightarrow& \#X\ge\frac{10}{11}N^d,
\end{eqnarray*}
concluding the proof.
\end{proof}
\subsubsection{Step 2: constructing the function $M$:}
\vspace{-2em}
\begin{figure}[ht]
\begin{minipage}[b]{.5\textwidth}
The definition of the function $M$ will heavily depend on a collection of sets, $\mathcal M$, with special properties. To define the collection $\mathcal M$ we will need the following definition:
Given $k$ an interval $I$ is called {\it a dyadic interval of order $k$} if there exists $j$ so that $I=\pm 2^{j\cdot k}+\left[0,2^k\right)$. A cube $\mathcal C$ is called {\it a dyadic cube of order $k$} if there are $d$ dyadic intervals of order $k$, $I_1,I_2,\cdots,I_d$ so that $\mathcal C=\prodit j 1 d I_j$. We say a cube $\mathcal C$ is {\it a dyadic cube} if there exists $k$ so that $\mathcal C$ is a dyadic cube of order $k$ (see Figure \ref{fig:dyadic} to the right).
\end{minipage}\hfill
\begin{minipage}[b]{.46\textwidth}
\centering
\includegraphics[scale=0.53]{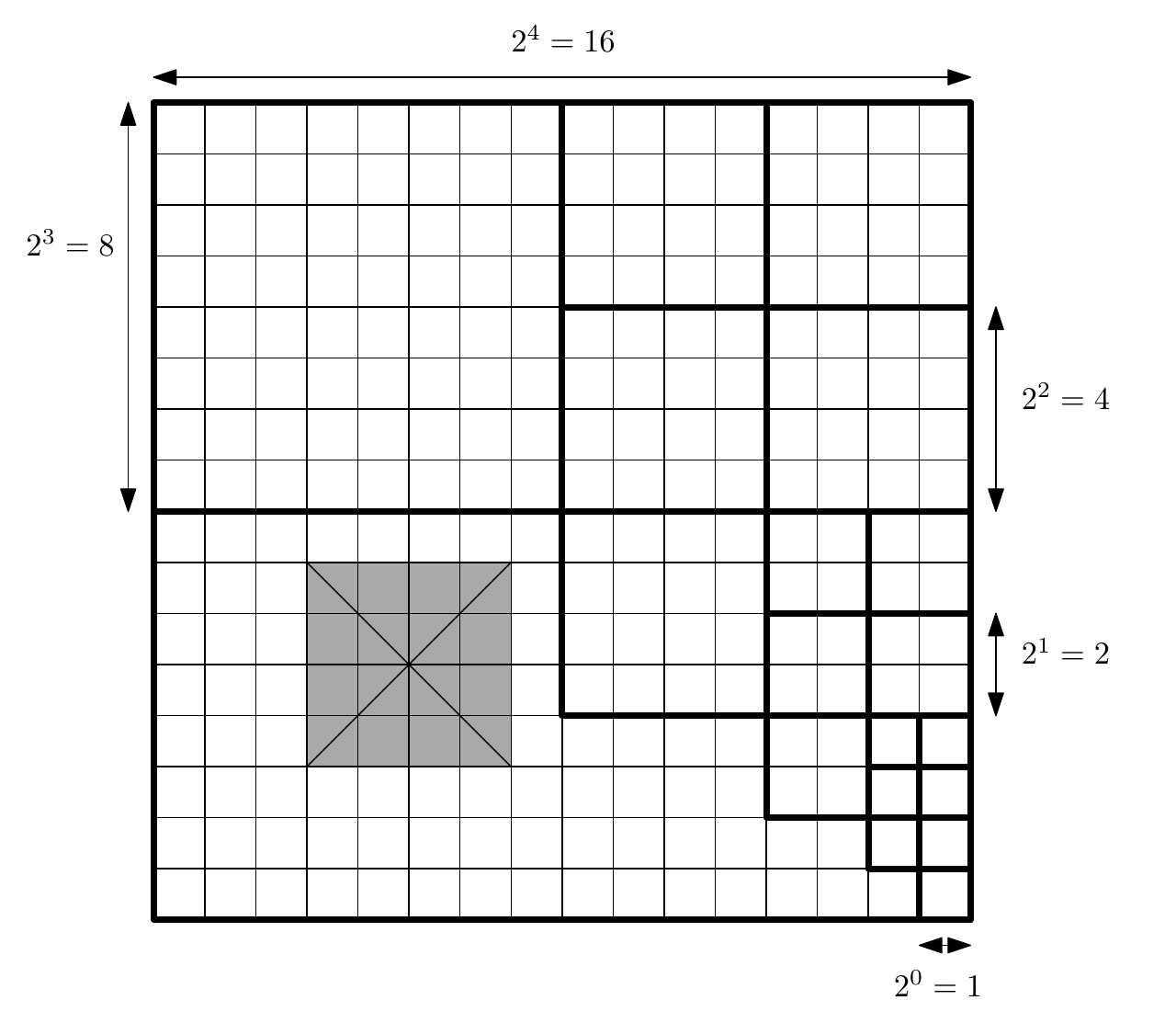}
\caption{This figure depicts dyadic cubes. The grey cube is not a dyadic cube even though it has edge-length 4.}
\label{fig:dyadic}
\end{minipage}
\end{figure}
\vspace{-1em}
\hspace{-1.8em}
Note that a dyadic cube of order $k$ is a disjoint union of $2^{d\bb{k-j}}$ dyadic cubes of order $j$, and that every two dyadic cubes of the same order are disjoint. This means we can define a partial order on the collection of dyadic cubes.
 
We assume without loss of generality that $N=2^{N_0}$, since for every other $N$ we can find $N_0$ such that  $2^{N_0}<N\le 2^{N_0+1}$, and the proof follows from this case with maybe a different constant $c_d$. For every $I\in\mathcal G$ we assign $\rho(I)=\underset{x\in I}\sup\;\rho(x)$. Then $\rho$, defined on elements of $\mathcal G$, assumes a finite number of values, and one may order the cubes $I\in\mathcal G$ as $I_1,\cdots,I_{N^d}$ in an ascending $\rho$ order, that is so that $\rho(I_k)\ge\rho(I_{k+1})$. For every $I_k\in\mathcal G$ we find a dyadic cube $J_k$ so that:
\begin{enumerate}
\item $I_k\subset J_k$.
\item $2\rho(I_k)\le \ell(J_k)< 4\rho(I_k)$.
\end{enumerate}
We define the collection $\mathcal M$ to be the collection of maximal elements $\bset{J_k}$, that is
$$
J_k\in\mathcal M\iff I_k\not\subset \bunion i 1 {k-1}J_i,
$$
for if $I_k\subset J_i$ for some $i<k$, then since $\rho(I_i)\ge\rho(I_k)$, $J_k\subseteq J_i$. $\mathcal M$ forms a cover for $Q$ and it is uniquely defined.\\
Let $n_\ell$ denote the number of elements in $\mathcal M$ with edge length $2^\ell$. We will define the function $M$ to be a step function, with steps
$$
s_m:=\min\bset{\bb{\frac{\alpha N^d}{\underset{\ell\ge m}\sum 2^{\ell} n_\ell}}^{\frac1{d-1}},\frac {N}{6d}},
$$
for some numerical constant $\alpha$ which we shall choose later. It is not hard to see that this is a monotone non-decreasing sequence. We denote the index of the first element satisfying $s_m= \frac {N}{6d}$ by $\bar m+1$ and stop defining the sequence there. Let $m_0$ be a constant, which will depend on the dimension, and determined later by Claim \ref{clm:large_sqrs_measure}, and define the function
$$
M(k):=	\begin{cases}
		2^{m_0}&, k\le s_{m_0}\\
		2^m&, s_{m-1}< k\le s_m,\;  m\in\bset{m_0+1,\cdots,\bar m+1}
		\end{cases}.
$$
Though the definition of $M$ may seem forced and unnatural now, we will soon see it is in fact very natural. To show that the function $M$ satisfies property (\ref{eq:M}), we relay on two key claims. The first claim bounds the total number of elements in $\mathcal G$ with large $\rho$. We will show that elements in $\mathcal M$ with large edge length do not cover a big part of $Q$:
\begin{claim}\label{clm:large_sqrs_measure}
There exist a numerical constant $C_1$ and an index $m_0$, which depend on the dimension $d$, so that
$$
\underset{m\ge{m_0}}\sum 2^{m\cdot d} n_m\le C_1\#\mathcal E.
$$
\end{claim}
\begin{proof}
Recall that the maximal function is defined by
$$
M_f(x):=\underset{B \text{ some ball}\atop\text{with } x\in B}\sup\frac1{m_d(B)}\underset B\int f(y)dm_d(y).
$$
Let $x$ be so that $\rho(x)> 2\sqrt d$, then by the way $\rho$ is defined
$$
\frac{m_d\bb{K\cap B\bb{x,\frac{\rho(x)}4}}}{m_d\bb{ B\bb{x,\frac{\rho(x)}4}}}<\delta_0,
$$
and, by definition, the maximal function $M_{\indic{Q\setminus K}}$ satisfies
\begin{eqnarray*}
M_{\indic{Q\setminus K}}(x)=\underset{B \text{ some ball}\atop\text{with } x\in B}\sup\frac1{m_d(B)}\underset B\int \indic{Q\setminus K}dm_d&=&\underset{B \text{ some ball}\atop\text{with } x\in B}\sup\bb{1-\frac{m_d(B\cap K)}{m_d(B)}}\\
&\ge&1-\frac{m_d\bb{K\cap B\bb{x,\frac{\rho(x)}4}}}{m_d\bb{ B\bb{x,\frac{\rho(x)}4}}}\ge 1-\delta_0\ge\frac12,
\end{eqnarray*}
for $\delta_0$ appropriately chosen. We note that the same holds for every $y\in B\bb{x,\frac{\rho(x)}4}$. We conclude that $B\bb{x,\frac{\rho(x)}4}\subset\bset{M_{\indic{Q\setminus K}}\ge\frac12}$. Using the maximal function theorem,
\begin{equation*}
\tag{$\dagger$}\label{eq:max-func}
m_d\bb{\underset{\bset{x ,\;\rho(x)>2\sqrt d}}\bigcup B\bb{x,\frac{\rho(x)}4}}\le m_d\bb{\bset{M_{\indic{Q\setminus K}}\ge\frac12}}\lesssim m_d(Q\setminus K)=\#\mathcal E,
\end{equation*}
where the last equality is by the definition of the set $K$.

Let $J\in\mathcal M$ be so that $\ell(J)>8\sqrt d$. Then there exists $x_J\in J$ with $\rho(x_J)\ge\frac{\ell(J)}4>2\sqrt d$, implying that $J\subset B(x_J,4\sqrt d\cdot\rho(x_J))$. We conclude that
$$
\frac{1}{16\sqrt d}J\subset B\bb{x_J,\frac{ \rho(x_J)}4},
$$
where $\frac{1}{16\sqrt d}J$ is the cube concentric with $J$, having edge-length $\frac{1}{16\sqrt d}\cdot\ell(J)$.

We choose $m_0$ to be the first integer satisfying that $2^{m_0}>8\sqrt d$. By using the inequality above, since the elements in $\mathcal M$ are disjoint,
\begin{eqnarray*}
\sumit m {m_0}\infty 2^{d\cdot m} n_m&=&\sumit m {m_0}\infty \bb{2^m}^d n_m=\underset{J\in\mathcal M\atop{\ell(J)\ge 2^{m_0}}}\sum m_d(J)=16^d\cdot d^{\frac d2}\underset{J\in\mathcal M\atop{\ell(J)\ge 2^{m_0}}}\sum m_d\bb{\frac{1}{16\sqrt d} J}\\
&=&16^d\cdot d^{\frac d2}m_d\bb{\underset{J\in\mathcal M\atop{\ell(J)\ge 2^{m_0}}}\bigcup  \frac{1}{16\sqrt d} J}\lesssim m_d\bb{\underset{J\in\mathcal M\atop{\ell(J)\ge 2^{m_0}}}\bigcup B\bb{x_J,\frac{\rho(x_J)}4}}\\
&\le&  m_d\bb{\underset{\bset{x ,\;\rho(x)>2\sqrt d}}\bigcup B\bb{x,\frac{\rho(x)}4}}\lesssim\#\mathcal E,
\end{eqnarray*}
by (\ref{eq:max-func}), concluding the proof.
\end{proof}
The second claim gives a softer condition to guarantee that property (\ref{eq:M}) holds.
\begin{claim}\label{clm:prop2}
There exists a numerical constant $C_2$ so that for every $c_0$ small enough if $\alpha=\frac1{12C_2}-c_0C_1>0$, then for all $m\ge m_0$, for every $k\le\bb{\frac{\alpha\cdot N^d}{\underset{\ell\ge m}\sum 2^{\ell} n_\ell}}^{\frac1{d-1}}$ at least $\frac{11}{12}$ of the basic cubes $I\in\mathcal G$ satisfy that $\left.\rho\right|_{A(I,k)\cap Q}\le 2^m$.
\end{claim}
It becomes clear now why the sequence $\bset{s_m}$ was defined in that particular way: Fix $k$ and let $m$ be so that $s_{m-1}<k\le s_m$, then by definition $M(k)=2^m$. Next,
$$
k\in K_I\iff \rho|_{A(I,k)\cap Q}\le M(k)=2^m.
$$
On the other hand,
$$
k\le s_{m}=\bb{\frac{\alpha\cdot N^d}{\underset{\ell\ge m}\sum 2^{\ell} n_\ell}}^{\frac1{d-1}}.
$$
Following Claim \ref{clm:prop2}, at least $\frac{11}{12}$ of the basic cubes $I\in\mathcal G$ satisfy that $\left.\rho\right|_{A(I,k)\cap Q}\le 2^m$, or, in other words, $\#\bset{I,\; k\in K_I}\ge\frac{11}{12}N^d$, and so the function $M$ satisfies property (\ref{eq:M}).
\begin{proof}[Proof of Claim \ref{clm:prop2}]
For every $J\in\mathcal M$, for every $x\in J$, $\rho(x)<\frac{\ell(J)}2$, by the way the collection $\mathcal M$ was defined. In particular, if $\ell(J)<2^m$, then for every $x\in J$, $\rho(x)<2^m$ as well.
We conclude that if $A(I,k)$ only intersects $J\in\mathcal M$ with $\ell(J)<2^m$, then every $x\in A(I,k)$ must satisfy $\rho(x)<2^m$, as $\mathcal M$ forms a cover for $Q$. Overall,
$$
\bset{I, \exists x\in A(I,k), \rho(x)>2^m}\subseteq\underset{J\in\mathcal M\atop \ell(J)>2^m}\bigcup\bset{I, A(I,k)\cap J\neq \emptyset}:=\underset{J\in\mathcal M\atop \ell(J)>2^m}\bigcup B(J,k).
$$
We will bound the number of elements in $B(J,k)$, i.e given $J\in\mathcal M$ and $k\in\bset{1,\cdots,\frac N{6d}}$, how many elements $I\in \mathcal G$ satisfy that $J\cap A(I,k)\neq\emptyset$?

To bound this quantity we will look at two cases. The first and simpler case is when $\ell(J)>2k+1$. Since $\ell(J)>2k+1$, the cube $J$ must contain at least one outer vertex of the set $A(I,k)$. We begin by choosing some order on the outer vertices of $A(I,k)$. Then we can associate each basic cube $I\in B(J,k)$ with the location of the first outer vertex of $A(I,k)$ lying inside $J$ (see Figure \ref{fig:sub_big} bellow). Because $J$ is a union of basic cubes and so is $A(I,k)$, the number of sites in $J$, in which such a vertex could be found, is equal to the number of basic cubes in $J$. We can therefore bound from above the number of elements in $B(J,k)$, in this case, by the number of possible vertices in $A(I,k)$ times the number of basic cubes in $J$, that is by $2^dm_d(J)=2^d\cdot\ell(J)^d$.

The second case is when $\ell(J)\le 2k+1$. In this case, if $J\cap A(I,k)\neq\emptyset$, then the boundary of $J$ must intersect $A(I,k)$. In particular, there are two adjacent vertices of $J$ one inside (or on) $A(I,k)$ and the other outside (or on) $A(I,k)$. We can identify every basic cube $I$ in $B(J,k)$ by indicating the basic cube on $A([0,1]^d,k)$ where the intersection of $A(I,k)$ and $J$ occurs, and the basic cube on the one dimensional facet connecting the two adjacent vertices of $J$ mentioned earlier  (see Figure \ref{fig:sub_small} bellow). If more than two vertices satisfy this, we choose an order on the collection of one dimensional facets and use the first one intersecting $A(I,k)$. We conclude that the number of elements in $B(J,k)$, in this case, is bounded by the number of basic cubes on all the one dimensional facets of $J$ times the number of basic cubes in $A([0,1]^d,k)$, which is bounded by
$$
2^d\cdot\ell(J)\cdot 6d\cdot(2k+1)^{d-1}\lesssim \ell(J)\cdot k^{d-1}.
$$
\begin{figure}[!ht]
	 \center
	    \begin{subfigure}[b]{0.49\textwidth}
		{
		  \centering
	      	\includegraphics[width=0.8\linewidth]{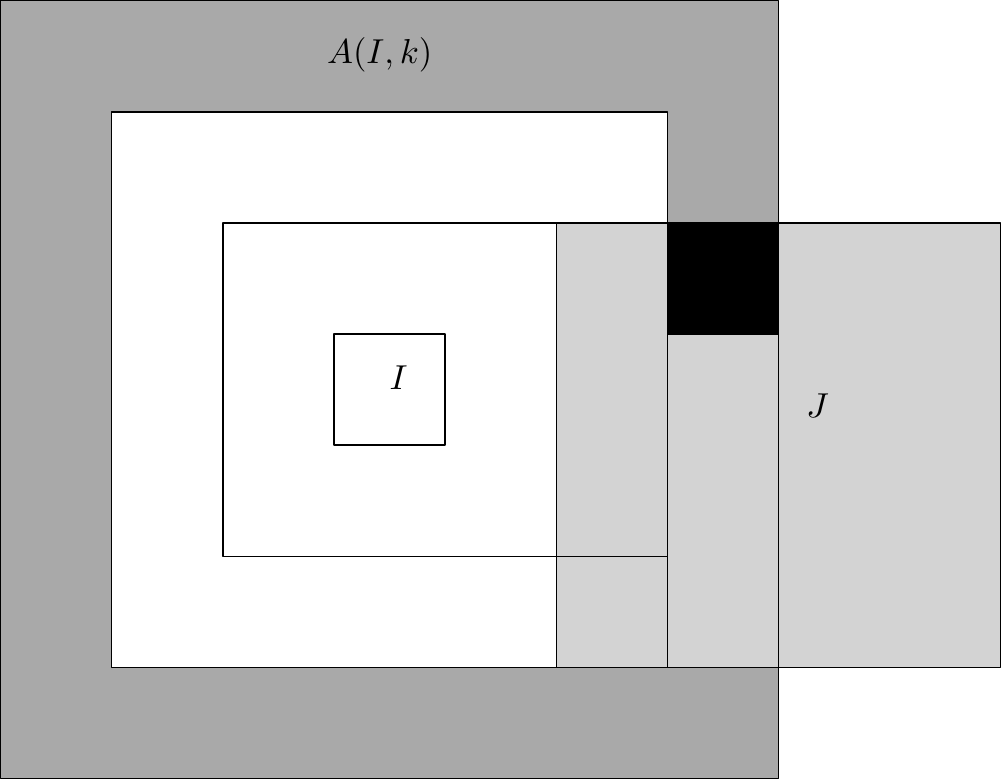}
		\caption{The case $\ell(J)\le2k+1$.}
		\label{fig:sub_small}
 		   }
	    \end{subfigure}
	    \begin{subfigure}[b]{0.49\textwidth}
		    {
		      \centering
	        	\includegraphics[width=0.8\linewidth]{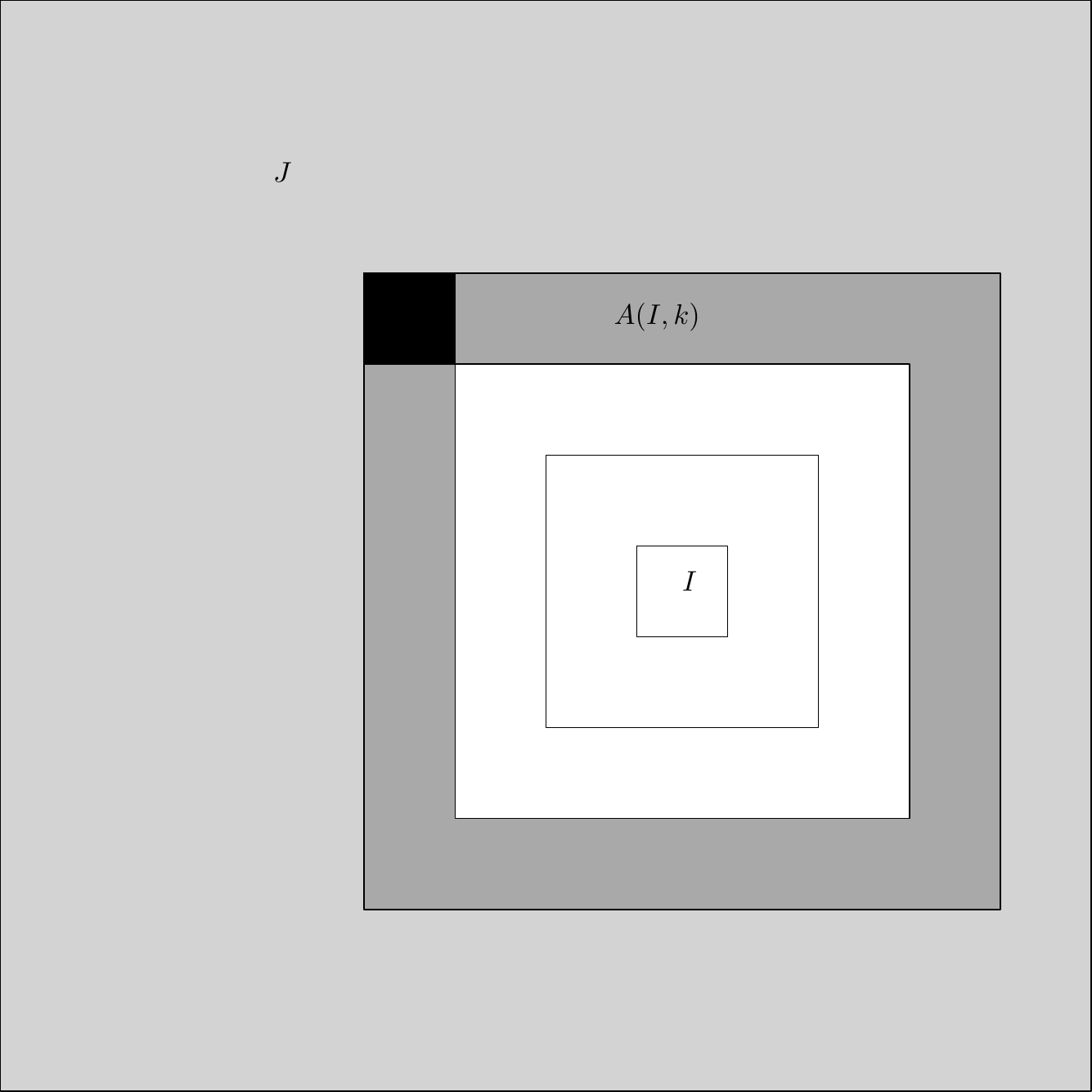}
		\caption{The case $\ell(J)>2k+1$. }
		 \label{fig:sub_big}
		    }
		    \end{subfigure}
	    \caption{This figure describes what happens in dimension $d=2$. In both figures, the black cube is the basic cube we choose to indicate $I\in B(J,k)$.}
	    \label{fig:clm2}
\end{figure}

Over all we conclude that
\begin{eqnarray*}
&&\#\bset{I,\;\exists x\in J\cap A(I,k) \text{ with }\rho(x)>2^m}\le  \underset{J\in\mathcal M\atop \ell(J)>2^m}\sum \#B(J,k)\lesssim\underset{J\in\mathcal M\atop \ell(J)>2^m}\sum \bb{\ell(J)^d+k^{d-1}\ell(J)}\\
&&\lesssim \underset{J\in\mathcal M\atop\ell(J)>2^m}\sum m_d(J)+k^{d-1}\underset{J\in\mathcal M\atop\ell(J)>2^m}\sum \ell(J)=\underset{\ell>m}\sum 2^{\ell\cdot d}n_\ell+k^{d-1}\underset{\ell>m}\sum 2^{\ell}n_\ell.
\end{eqnarray*}
Combining this with Claim \ref{clm:large_sqrs_measure}, we see that if $k^{d-1}\le\frac{\alpha N^d}{\underset{\ell\ge m}\sum 2^{\ell} n_\ell}$ and $\#\mathcal E\le c_0N^d$, then there exists a constant $C_2$, which depends on the dimension alone, so that
\begin{eqnarray*}
&&\#\bset{I,\;\exists x\in Q\cap A(I,k) \text{ with }\rho(x)>2^m}\le \cdots\le  C_2\underset{\ell> m}\sum 2^{\ell\cdot d}n_\ell+C_2k^{d-1}\underset{\ell\ge m}\sum 2^{\ell}n_\ell \\
&&\le C_2\cdot C_1\#\mathcal E+C_2\alpha\cdot N^d\le N^d\bb{C_2C_1c_0+C_2 \alpha}=\frac {N^d}{12},
\end{eqnarray*}
by the way we defined $\alpha$. We therefore get that
\begin{eqnarray*}
\#\bset{I,\;\forall x\in Q\cap A(I,k) \text{ with }\rho(x)\le2^m}=N^d-\#\bset{I,\;\exists x\in Q\cap A(I,k) \text{ with }\rho(x)>2^m}\ge N^d-\frac {N^d}{12}=\frac{11}{12}N^d,
\end{eqnarray*}
concluding the proof.
\end{proof}
\subsubsection{Step 3: Bounding $\sumit k 1 {\frac {N}{6d}}\frac1{M(k)}$ from bellow:}
If $s_{m_0}=\frac N{6d}$, then by the way the function $M$ is defined, for all $k$ we have $M(k)=2^{m_0}$, implying that for some constant which depends on the dimension alone,
$$
\sumit k 1 {\frac {N}{6d}} \frac1{M(k)}=\sumit k 1 {\frac {N}{6d}} \frac1{2^{m_0}}= c_d\cdot N,
$$
which yields a smaller bound (up to multiplication by a constant) than the one indicated in the lemma. Otherwise, assume that $s_{m_0}<\frac N{6d}$. Since the function $M$ was defined as a step function,
\begin{eqnarray*}
\sumit k 1 {\frac {N}{6d}} \frac1{M(k)}&=&\frac{s_{m_0}}{2^{m_0}}+\sumit m {m_0+1} {\bar m+1}2^{-m}\bb{s_m-s_{m-1}}=\frac{s_{m_0}}{2^{m_0}}+\sumit m {m_0+1} {\bar m+1}2^{-m}s_m-\sumit m {m_0+1} {\bar m+1}2^{-m}s_{m-1}\\
&=&\sumit m {m_0} {\bar m+1}2^{-m}s_m-\frac12\sumit m {m_0} {\bar m}2^{-m}s_m=\frac12\sumit m {m_0} {\bar m}2^{-m}s_m+2^{-\bar m-1}s_{\bar m+1}\sim N\cdot 2^{-\bar m}+\sumit m {m_0} {\bar m}\frac{s_m}{2^m}.
\end{eqnarray*}
To bound $\sumit m {m_0} {\bar m}\frac{s_m}{2^m}$ from bellow, we use Jensen's inequality with the convex function $g(t)=t^{-\frac1{d-1}}$:
\begin{eqnarray*}
\sumit m {m_0}{\bar m}\frac{s_m}{2^m}&=&\sumit m {m_0}{\bar m}g\bb{\frac{2^{m(d-1)}}{s_m^{d-1}}}=\bb{\bar m-m_0+1}\frac{\sumit m {m_0}{\bar m}g\bb{\frac{2^{m(d-1)}}{s_m^{d-1}}}}{\bar m-m_0+1}\\
&\ge&\bb{\bar m-m_0+1}g\bb{\frac{\sumit m {m_0}{\bar m}\frac{2^{m(d-1)}}{s_m^{d-1}}}{\bar m-m_0+1}}= \frac{\bb{\bar m-m_0+1}^{1+\frac1{d-1}}}{\bb{\sumit m {m_0} {\bar m}\frac{2^{m(d-1)}}{s_m^{d-1}}}^{\frac1{d-1}}}.
\end{eqnarray*}
To bound $\sumit m {m_0} {\bar m}\frac{2^{m(d-1)}}{s_m^{d-1}}$ from above, we note that by definition 
$$
s_m=\bb{\frac{\alpha\cdot N^d}{\underset{\ell\ge m}\sum 2^{\ell} n_\ell}}^{\frac1{d-1}}.
$$
This implies that
\begin{eqnarray*}
\sumit m{m_0}{\bar m}\frac{2^{m(d-1)}}{s_m^{d-1}}=\frac1{\alpha N^d}\sumit m{m_0}{\bar m} 2^{m(d-1)}\underset{\ell\ge m}\sum 2^{\ell} n_\ell=\frac1{\alpha N^d}\underset{\ell\ge m_0}\sum 2^{\ell} n_\ell\sumit m{m_0}\ell 2^{m(d-1)}\sim \frac1{\alpha N^d}\underset{\ell\ge m_0}\sum 2^{d\cdot \ell} n_\ell\sim\frac{\#\mathcal E}{N^d},
\end{eqnarray*}
following Claim \ref{clm:large_sqrs_measure}. Then
\begin{eqnarray*}
\sumit m {m_0}{\bar m}\frac{s_m}{2^m}&\ge& \frac{\bb{\bar m-m_0+1}^{1+\frac1{d-1}}}{\bb{\sumit m {m_0} {\bar m}\frac{2^{m(d-1)}}{s_m^{d-1}}}^{\frac1{d-1}}}\gtrsim \bar m^{1+\frac1{d-1}}\cdot\bb{\frac{N^d}{\#\mathcal E}}^{\frac1{d-1}},
\end{eqnarray*}
implying that
\begin{eqnarray*}
\sumit k 1 {\frac {N}{6d}} \frac1{M(k)}\sim2^{-\bar m}\cdot N+\sumit m {m_0}{\bar m}\frac{s_m}{2^m}\gtrsim 2^{-\bar m}\cdot N+\bar m^{1+\frac1{d-1}}\cdot\bb{\frac{N^d}{\#\mathcal E}}^{\frac1{d-1}}.
\end{eqnarray*}
We would like our lower bound to hold for every $\bar m$. To find the optimal inequality, we define the function
$$
\varphi(x):=2^{-x}+x^{1+\frac1{d-1}}\cdot\bb{\frac{N}{\#\mathcal E}}^{\frac1{d-1}}.
$$
As we know nothing about $\bar m$, we will look for the absolute minimum of $\varphi$ in $\left[1,\frac N{6d}\right]$ and use whatever inequality this minimum satisfies:
$$
\varphi'(x)=-2^{-x}\log(2)+\bb{1+\frac1{d-1}}\bb{\frac{N}{\#\mathcal E}}^{\frac1{d-1}}\cdot x^{\frac1{d-1}}=0\iff x^{\frac1{d-1}}2^x\sim\bb{\frac{\#\mathcal E}{N}}^{\frac1{d-1}}
$$
and the later is a minimum point since $\varphi''\ge 0$. The lower bound this minimum produces is
$$
\exp\bb{-c_d\cdot\frac{N}{1+\bb{\frac{\#\mathcal E}{N}}^{\frac1{d-1}}}\log^{\frac d{d-1}}\bb{2+\frac{\#\mathcal E}{N}}},
$$
concluding the proof of the Main Lemma.
 \end{proof}

\section{Upper bound- example}\label{sec:upper}
\subsection{Prologue}
If $f(t)\lesssim t$, then the upper bound we are looking to prove is exponential. Let
$$
W(x_1,\cdots,x_d):=\cos(2\pi x_1)\prodit j 2 d\cosh\bb{\frac{2\pi x_j}{\sqrt{d-1}}}\indic{\bset{\abs{x_1}\le\frac14}}(x).
$$
This function is subharmonic, non-negative and supported on the set $\bset{\abs {x_1}<\frac14}$. By taking the maximum over translations and rotations of this function, we create a subharmonic function $u$ so that for every basic cube, $I$, we have $m_d(Z_u\cap I)\ge\frac12$, while $M_u(I)\ge1$, i.e $u$ oscillates in every basic cube. On the other hand, the growth of the function $W$ is exponential, implying that the growth of the function $u$ is exponential as well, concluding the proof of the case $f(t)\lesssim t$. From now on we will assume that $f(t)\gg t$.

Though the construction, which we will eventually present in this section, will be applicable for all the dimensions, there is an essential difference between the case of the plane and higher dimensions.

We would have loved to use a construction similar to the one presented in \cite{Us2017} of a `self similar' subharmonic function. Two issues arise when one tries this approach. The first issue lies in the very essence of `self similarity'. The whole idea of `self similarity' means that for every compact set $K$ and for every cube $[-\ell,\ell]^d$, there are $\sim\ell^d$ disjoint copies of the set $K$ where the function $u$ is defined in the same way, i.e there is some uniform $\delta>0$ and a set $\bset{\omega_j}_{j=1}^{\delta\cdot\ell^d}\subset [-\ell,\ell]^d$ so that $\omega_j+K\cap\omega_m+K=\emptyset$ whenever $j\neq m$, while for every $z\in K$ and $j$ we have $u(z)=u(z+\omega_j)$. This implies that if there is one rogue basic cube, then for every $\ell$ large enough there are $\sim \ell^d$ rogue basic cubes in $[-\ell,\ell]^d$. For every function $f$, satisfying $f(t)\ll t^d$, we accumulate too many rogue basic cubes using this method, and even for $f(t)\sim t^d$ we do not get the optimal bound on the growth using this method.

The second issue arises only in dimensions higher than $3$. The construction in \cite{Us2017} uses functions similar to the function $W$, presented in the proof of the case $f(t)\lesssim t$. The support of these functions is of the form hyperspace $\times$ relatively small interval. We will refer to these sets as `walls'. These `walls' are used to weld together copies of the function $u$ on some compact sets, to get the `self similarity' property. The number of basic cubes lying in the intersection of such a `wall' and a cube of edge length $N$ is $N^{d-1}\times$ (the length of the relatively small interval). As we are considering the case $f(t)\gg t$, this does not form a problem for the plane. In higher dimensions, this becomes a problem whenever $f(t)\ll t^{d-1}$, and in fact for any function $f(t)\ll t^d$ this method will not produce the optimal bound on the growth.

We will define a subharmonic function whose support looks like an upside down tree. Instead of using these `wall' sets, we will use tube-like sets. The entire support of the function $u$ will be a union of tubes of different diameters. This solves the second issue we discussed. To address the first issue, instead of duplicating the same function over and over, we will create similar but different functions, which have less and less rogue basic cubes. We formally describe the support of this function bellow. The idea of trees (explained in Subsection \ref{sec:Tree} bellow) was inspired by numerous examples in Jones and Makarov's paper, \cite[Section 6]{JoneMakarov1995}.

\subsection{The set} \label{sec:Tree}
We begin our discussion with a description of the set, on which the subharmonic function we create, is supported (in fact, it will be supported on an extremely similar set, but not exactly this set). This set is a union of $2^d$ disjoint sets, each of them will be kind of an upside down tree described from its tree-top to the trunk.

Let $\bset{v_j}$ denote the vertices of the cube $[-1,1]^d$. For every scalar $\alpha$ we write $\alpha\cdot v_j$ to describe the rescaling of the point $v_j$ by $\alpha$. For example, let $v_0=(1,1,\cdots,1)$, then $\alpha \cdot v_0=(\alpha,\alpha,\cdots,\alpha)$. Let $R_j$ denote the linear map taking  the cube $v_0+[-1,1]^d$ to $v_j+[-1,1]^d$ with $R_j(0)=0$ and $R_j(v_0)=v_j$. The set we describe here will be a union of disjoint copies of one set $T\subset (0,\infty)^d$, that is $\widetilde T:=\bunion j 1 {2^d-1}R_j(T)$. It is therefore enough to describe the set $T$, which we shall do using a recursive description.

\paragraph{Basic terminology and the first steps} A {\it tube of diameter $\delta$} is any translation or rotation of the set
$$
\bset{x=(x_1,\cdots,x_d),\; \forall j\ge 2,\;\abs{x_j}<\frac\delta2,\text{ and } x_1\in[a,b]}
$$

\begin{figure}[!h]
\begin{minipage}[c]{0.58\columnwidth}
where $a<b$ are both real numbers. Given a cube $Q$ we will denote by $\ell(Q)$ the edge length of $Q$.\\
For every basic cube lying inside $[0,2)^d$, we connect the center of the basic cube, $v^*$, with $v_0$ by a tube of diameter $\eps_1$. We place the tube so that the centerline of the tube is the line connecting $v_0$ and $v^*$. We will call this structure a {\it basic subtree of leaves diameter $\eps_1$} and denote it $T_1$. For illustration in dimension $d=3$ see Figure \ref{fig:stepone} to the right.
\end{minipage}%
\hfill
\begin{minipage}[c]{0.4\columnwidth}
	\centering
	\includegraphics[width=0.65\textwidth]{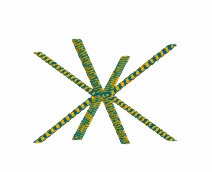}
	\caption{A basic subtree in dimension $d=3$, composed of $8=2^d$ tubes connecting the basic cubes surrounding $v_0$ with $v_0$.}
	\label{fig:stepone}
\end{minipage}
\end{figure}

On step 2, we note that $[0,4)^d$ is composed of $2^d$ disjoint dyadic cubes of order 1, i.e of edge length $2$. In one of them, $[0,2)^d$, we already defined a basic subtree of leaves diameter $\eps_1$. We shall call it {\it the inner subtree of rank 1}. In each of the other $2^d-1$ cubes we create a basic subtree of leaves diameter $\eps_2$ in a similar manner to step 1, for some $\eps_2$.  We will call these {\it the outer subtrees of rank 1}. We then connect the center of each dyadic cube of order 1 to $2v_0$ by a tube of relative diameter $\eps_2$, i.e the diameter of the tube is $4\eps_2$. These tubes will be called  brunches. More precisely, a {\it brunch} will be any tube of diameter bigger than $2\eps_1$. We call this structure a {\it subtree of rank 2} and denote it $T_2$. For illustration examples in dimensions 2 and 3, see Figures \ref{fig:steps} and \ref{fig:steps3d} bellow.

\begin{figure}[!h]
\centering
\includegraphics[scale=0.6]{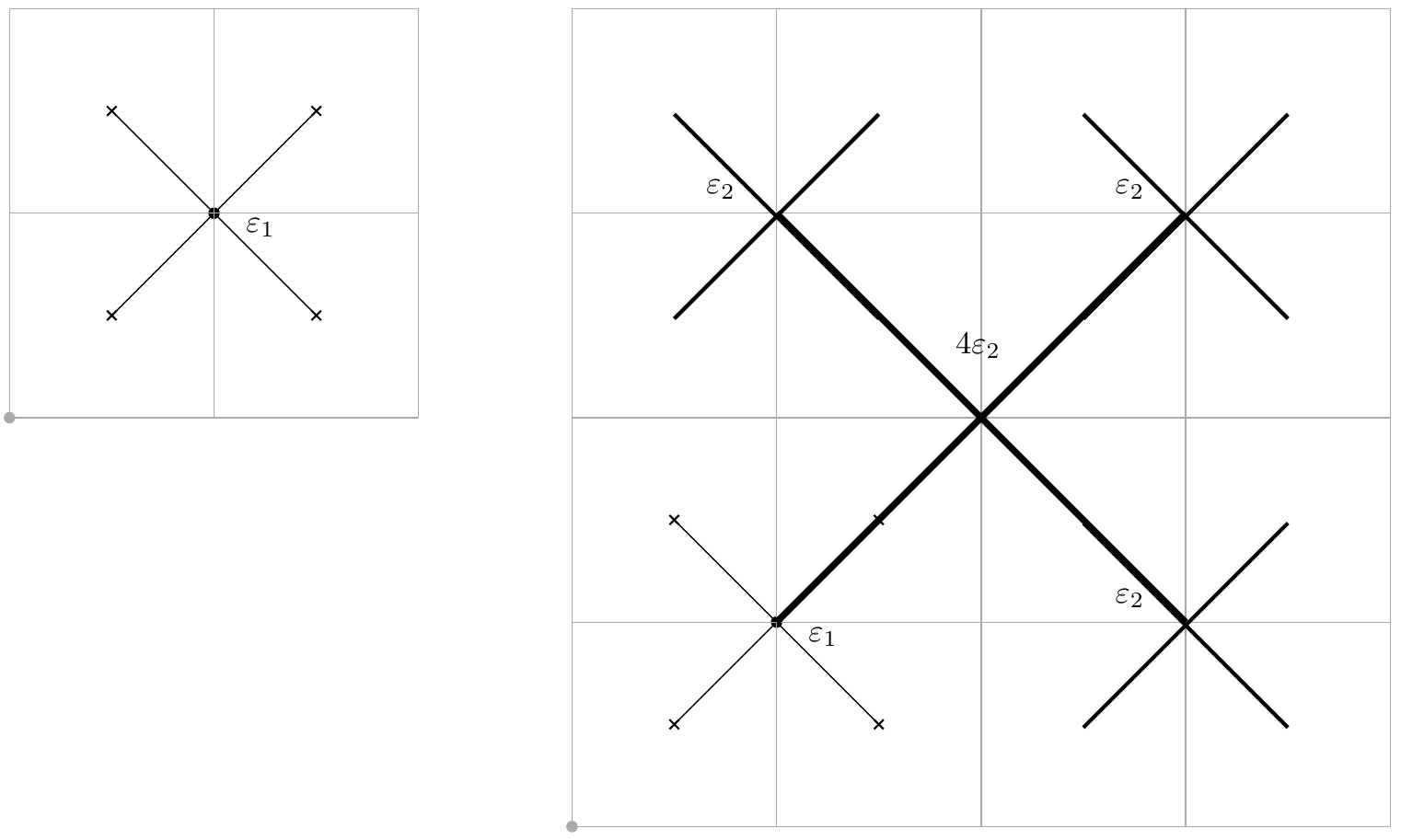}
  \caption{This figure depicts the first two steps in dimension $d=2$. The figure to the left depicts the first step and the figure to the right depicts the second step. The grey lines indicate the boundary of the basic cubes. The numbers next to the tubes indicate their diameter.}
  \label{fig:steps}
\end{figure}

\begin{figure}[!ht]
\noindent
\begin{minipage}[b]{0.397\columnwidth}
	 \centering
	    \begin{subfigure}[b]{\columnwidth}
		    {
	        	\includegraphics[width=\textwidth]{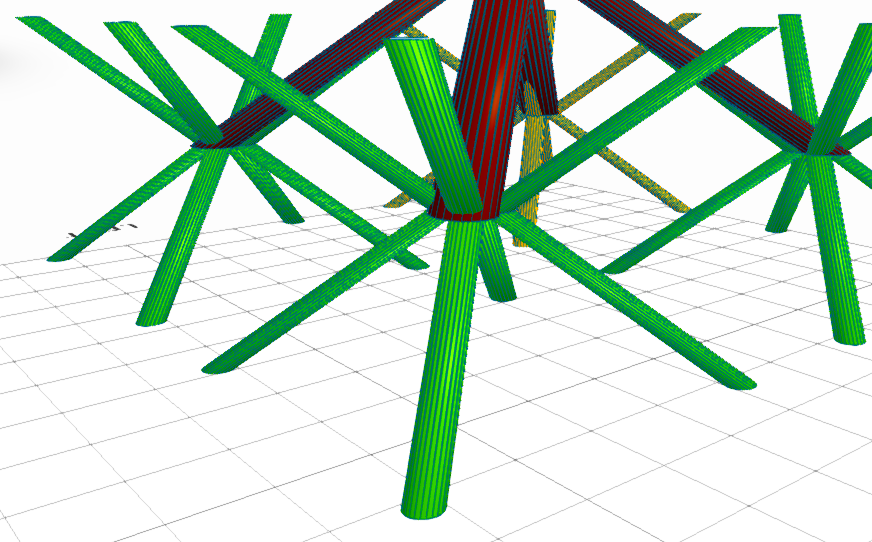}
		 \caption{Closeup on a subtree of rank two.}
		 \label{fig:steptwo_brunch_closeup}
		    }
		    \end{subfigure}
	   \begin{subfigure}[b]{\columnwidth}
		    {
	        	\includegraphics[width=\textwidth]{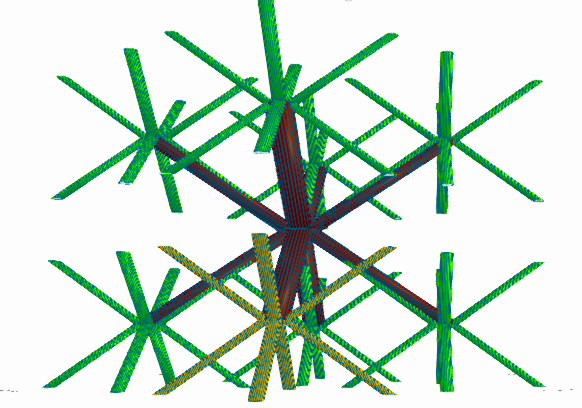}
		 \caption{A subtree of rank two structure.}
		 \label{fig:steptwo_structure}
		    }
		    \end{subfigure}
\end{minipage}%
\hfill
\begin{minipage}[b]{0.59\columnwidth}
\centering
	    \begin{subfigure}[b]{\columnwidth}
		{
	      	\includegraphics[width=\textwidth]{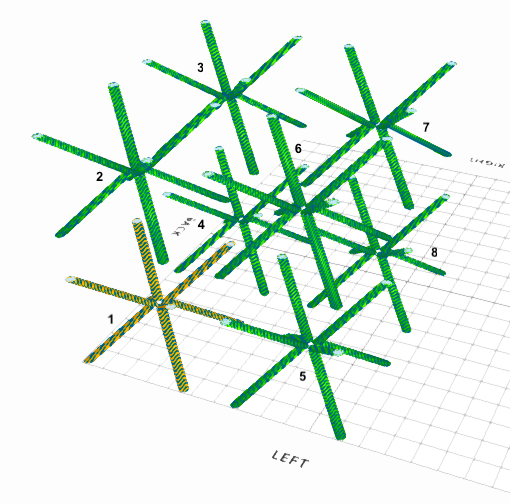}
		\caption{A collection of basic subtrees for step two.}
		\label{fig:steptwo_leaves}
 		   }
		   \end{subfigure}\\
\end{minipage}%

	    \caption{This figure depicts the first 2 steps in dimension $d=3$. The basic subtrees in Figure (\ref{fig:steptwo_leaves}) describes all the basic subtrees that participate in the second step. They are enumerated to indicate they have different diameters. The first one is the basic subtree from the first step appearing in Figure \ref{fig:stepone}. It has diameter $\eps_1$ while the others have diameter $\eps_2$. Figure (\ref{fig:steptwo_brunch_closeup}) zooms in on how the tubes which are considered `brunches' (when they are wide enough) are connected to the subtrees of rank 1.}
	    \label{fig:steps3d}
\end{figure}

\paragraph{Constructing an outer subtree of rank $(k+1)$} In general, we would like the outer subtree of rank $(k+1)$ to have two kinds of brunches diameters, corresponding to two stages of the tree.
\newpage Given $k$ we will describe the outer subtree of rank $(k+1)$ centered at $2^kv_0$. For all the other outer trees of rank $(k+1)$ we will use translations and rotations of this structure. Define $\eps_k:=2^{s_k-k}$ for an integer $s_k$ satisfying that 
$$
\bb{\frac{f(2^k)}{2^k}}^{\frac1{d-1}}\le s_k^{\frac1{d-1}}2^{s_k}\le4\bb{\frac{f(2^k)}{2^k}}^{\frac1{d-1}}.
$$
We will construct the outer subtree of rank $(k+1)$, $\widetilde{T_{k+1}}$, starting from the outside in, i.e starting from the big brunches. The cube $2^kv_0+[-2^k,2^k)^d$ is composed of $2^d$ dyadic cubes of order $k$, i.e of edge length $2^k$. We connect the center of each of these cubes to $2^kv_0$ by tubes of relative diameter $\eps_k$, i.e tubes of diameter $2^k\eps_k$. We place these tubes so that their centerlines are the lines connecting $2^kv_0$ with the center of these dyadic cubes. Next, every dyadic sub-cube of order $k$ is composed of $2^d$ dyadic sub-cubes of order $k-1$. We connect the center of each such sub-cube with the center of the dyadic cube of order $k$ containing it by a tube of relative diameter $\eps_k$. We continue this process for $s_k$ steps. We continue to connect dyadic cubes until we reach the cubes of order 0, these are basic cubes. Unlike the first $s_k$ steps, on the consecutive steps we use tubes of absolute diameter $\eps_1$. We denote this structure by $\widetilde{T_{k+1}}$.

\paragraph{The general step} To create the subtree of degree $k+1$ centered at $2^{k}v_0$, we take the inner subtree of rank $k$ we already created, $T_k$, and $2^d-1$ translated and rotated copies of $\widetilde{T_k}$ constructed above as the outer trees of degree $k$ centered at $2^kv_0+2^{k-1} v_j$ for $j\in\bset{1,\cdots,2^d-1}$. We then connect the center of each subtree to $2^{k}v_0$ using a tube of relative diameter $\delta_k:=\bb{\frac{f(2^k)}{2^{k\cdot d}}}^{\frac1{d-1}}$. We let $T$ denote the limiting set, i.e $T=\bunion k 1 \infty T_k$.

\paragraph{Properties of the set $T$} Given a basic cube, $I$, we say the set $T$ is {\it sparse in $I$} if 
$$
m_d(I\setminus T)>1-\sqrt d\cdot (2\eps_1)^{d-1}>\frac12,
$$
(the latter holds for $\eps_1$ small enough). We point out that every basic cube intersects the tree $T$, for if $I$ intersects a brunch then in particular it intersects $T$. Otherwise, it intersects a basic subtree, and those have leaves of diameter $\eps_1$, meaning in this case, $T$ is sparse in $I$.
\begin{claim}\label{clm:bad_guys}
For every $k$
$$
\#\bset{\text{basic cubes in }[0,2^k)^d\atop \text{in which }T\text{ is NOT sparse}}\lesssim f(2^k).
$$
\end{claim}
\begin{proof}
We note that $T$ is sparse in every basic cube which does not intersect a brunch. To bound the number of basic cubes in which $T$ is not sparse, we can therefore bound the number of basic cubes which intersect brunches. In addition, by the way the tree $T$ is constructed, it is enough to to look at $T_{k+1}$ to make such estimates for the cube $[0,2^k)^d$.\\
We begin by bounding the number of basic cubes intersecting brunches in $\widetilde{T_j}$, an outer subtrees of rank $j$. By the way $\widetilde{T_j}$ was constructed, all the brunches originate from the first $s_j$ steps of the construction:
\begin{eqnarray*}
&&\#\bset{\text{basic cubes intersecting brunches in }\widetilde{T_j}}\sim_d\; m_d\bb{\bset{\text{brunches in }\widetilde T_j}}\\
&&=\sumit \ell 0 {s_j-1} \#\bset{\text{dyadic cubes of order }(j-\ell)\text{ in }\widetilde{T_j}}\cdot \text{(tube length on step }\ell)\cdot\bb{\text{tube diameter on step }\ell}^{d-1}\\
&\sim&\sumit \ell 0 {s_j-1} 2^{d\cdot\ell}\cdot2^{j-\ell}\cdot\bb{\eps_j\cdot 2^{j-\ell}}^{d-1}=2^{j\cdot d}\eps_j^{d-1}\cdot s_j=2^{j\cdot d}\cdot  \bb{\frac{2^{s_j}}{2^j}}^{d-1}\cdot \bb{s_j^{\frac1{d-1}}}^{d-1}\\
&=&2^{j}\bb{2^{s_j}\cdot s_j^{\frac1{d-1}}}^{d-1}\le 2^{j}\cdot 4^{d-1}\cdot\frac{f(2^j)}{2^j}=4^{d-1}f(2^j),
\end{eqnarray*}
by the way the sequence $\bset{s_j}$ was chosen. Now, $f$ is regularly varying, therefore there exists a slowly varying function $g$ so that $f(t)=t^\alpha g(t)$, where $\alpha\in[1,d]$, as $f(t)\gg t$. This implies that for every $t\ge t_0$ large enough
$$
\frac{f(t)}{f(2t)}=\frac1{2^\alpha}\cdot\frac{g(t)}{g(2t)}\in\bb{\frac 1{2^{d+1}},\frac23}.
$$
We will now bound the number of basic cubes in $[0,2^k)^d$ in which $T_{k+1}$ is not sparse. Let $j_0$ be the first integer satisfying $2^{j_0}>t_0$, and let $b_j$ denotes the number of basic cubes intersecting a brunch in $T_j$. Using the estimate we showed for outer trees,
\begin{eqnarray*}
b_k&=&b_{k-1}+\#\bset{\text{basic cubes in the `handle'}\atop{\text{connecting the subtrees}}}+\bb{2^d-1}\#\bset{\text{basic cubes intersecting brunches}\atop\text{in an outer subtree of rank }(k-1)}\\
&\lesssim_d& b_{k-1}+2^k\cdot\bb{2^k\cdot\delta_k}^{d-1}+f(2^{k-1})=b_{k-1}+2^{k\cdot d}\cdot \frac{f(2^k)}{2^{k\cdot d}}+f(2^{k-1})\le b_{k-1}+2f(2^k)\lesssim\sumit j 1 k f(2^j)\\
&=&\sumit j 1 {j_0} f\bb{2^j}+ f\bb{2^k}\sumit j {j_0+1}{k} \frac{f(2^j)}{f\bb{2^k}}= \sumit j 1 {j_0} f\bb{2^j}+f\bb{2^k}\sumit j{j_0+1}k \prodit \ell j {k-1}\frac{f(2^\ell)}{f\bb{2^{\ell+1}}}\\
&\le&\sumit j 1 {j_0}f(2^j)+f\bb{2^k}\sumit j {j_0+1}k \bb{\frac23}^{k-j}\sim f\bb{2^{k}},
\end{eqnarray*}
since we saw that $\frac{f(t)}{f(2t)}\le\frac23$. This concluds the proof, since the number of basic cubes in $[0,2^k)^d$, in which $T_{k+1}$ is not sparse, is bounded by $b_k$ plus the basic cubes intersecting the `handle' connecting $T_k$ to $T_{k+1}$, which intersects $\lesssim_d\; f(2^{k+1})\sim_d\; f(2^k)$ basic cubes.
\end{proof}
\subsection{The function}
In this subsection we shall construct the desired subharmonic function. The support of this function will be closely related to the `tree set' described in subsection \ref{sec:Tree} above. The first step is to construct a subharmonic function which is supported on half tubes.

For every $\eps>0$ we define the function $T_\eps:\R^d\rightarrow\R$ by
$$
T_\eps(x)=T_\eps(x_1,x_2,\cdots,x_d):=\cosh\bb{\frac{\pi\sqrt{d-1}}\eps\cdot x_1}\prodit j 2 d\cos\bb{\frac\pi\eps\cdot x_j}\cdot\indic{\bset{\abs{x_j}<\frac\eps2}}(x),
$$
for $x=(x_1,\cdots,x_d)\in\R^d$. The set $\bset{T_\eps\neq 0}$ is an infinite tube of diameter $\eps$ whose centerline is the $x_1$-axis. The function $T_\eps$ is subharmonic as its laplacian is 0 on $\R^d\setminus \bunion j 2 d \bset{\abs{x_j}=\frac\eps 2}$ and along the partial hyper-planes composing the boundary of this set,  it is continuous and satisfies the mean value property.

We define the function
$$
L_\eps(x)=	\begin{cases}
				\max\bset{T_\eps(x)-1,0}&, x_1\ge 0\\
				0&, \text{otherwise}
			\end{cases}.
$$
Figure \ref{fig:leaf} bellow shows the support of the function $L_{\frac\pi4}$, i.e the set $\bset{L_{\frac\pi4}>0}$.
\begin{figure}[!ht]
\centering
\includegraphics[width=0.35\textwidth]{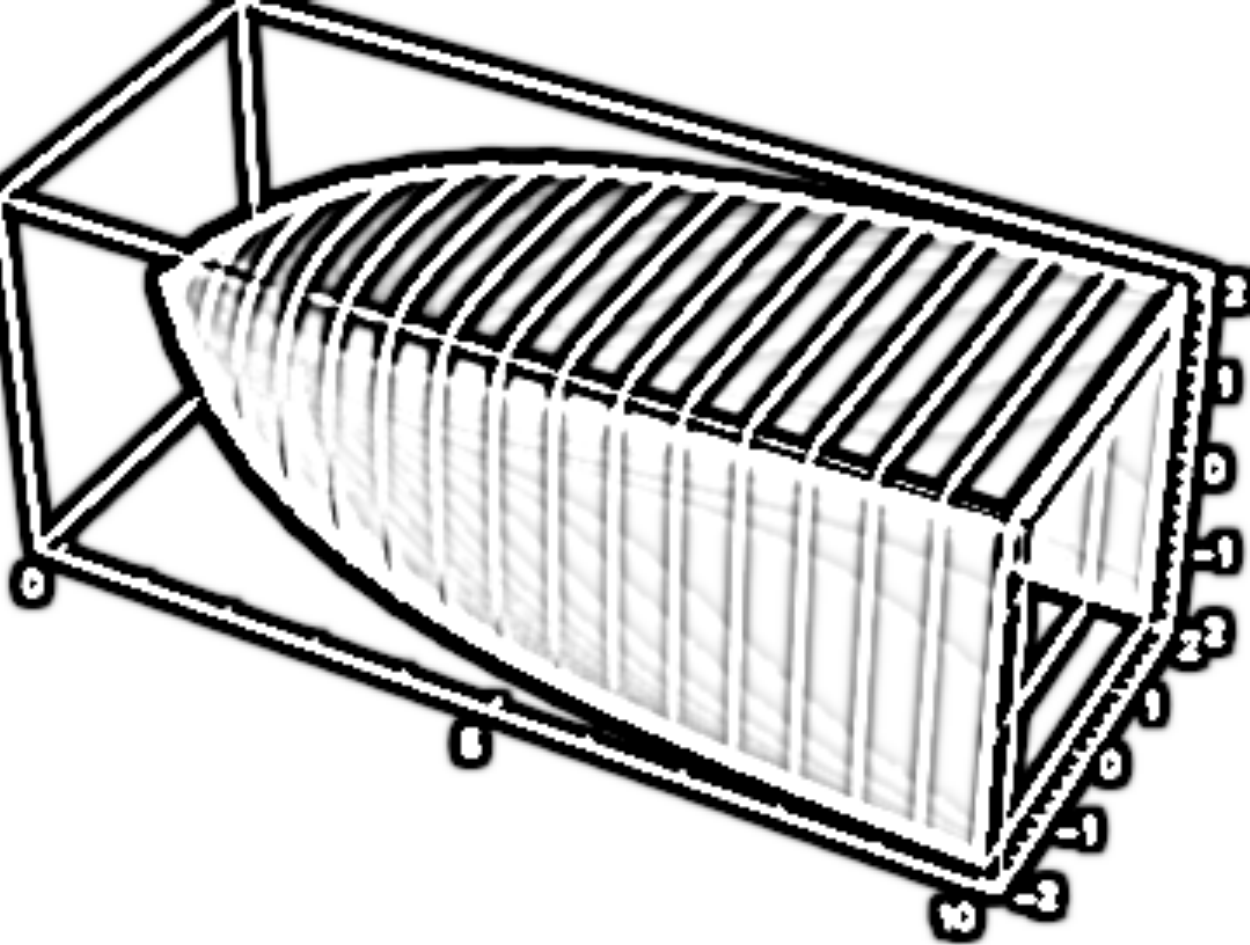}
  \caption{This figure shows the surface $\bset{L_{\frac\pi4}>0}$ in dimension 3.}
  \label{fig:leaf}
\end{figure}
\paragraph{Properties of the function $L_\eps$}
\begin{enumerate}
\item The function $L_\eps$ is continuous and subharmonic as locally it is either zero or the maximum between two subharmonic functions. 
\item \label{L_eps:low_bnd} Define the set
$$
G_\eps:= \sbb{\bintersect j 2 d\bset{\abs{x_j}\le\frac\eps 3}}\cap\bset{\abs{x_1}\ge\frac{\eps\cdot d\log 2}{\pi\sqrt{d-1}}}.
$$
The set $G_\eps$ is a closed set, and since for every $\abs t<\frac\eps 3$,  $\cos\bb{\frac{\pi}\eps\cdot t}\ge\cos\bb{\frac{\pi}\eps\cdot\frac\eps3}=\frac12$, every $x\in G_\eps$ satisfies
$$
L_\eps(x)\ge 2^{-(d-1)}\cosh\bb{\frac{\pi\sqrt{d-1}}\eps\cdot x_1}> 2^{-(d-1)}\frac{\exp\bb{\frac{\pi\sqrt{d-1}}\eps\cdot \abs{x_1}}}2> 2^{-d}\exp\bb{\log\bb{2^d}}=1.
$$
\item \label{L_eps:zero_set}The set $\bset{L_\eps>0}$ is contained in a tube of diameter $\eps$ interseced with $\bset{x_1>0}$.
\item \label{L_eps:up_bnd}Given a bounded set $E\subset\R^d$, $M_{L_\eps}(E):=\underset{x\in E}\sup\; L_\eps(x)\le \exp\bb{\frac{\pi\sqrt{d-1}}\eps\cdot \underset{x\in E}\sup\; \abs{x_1}}$.
\end{enumerate}
We are now ready to construct the function. Our function will be a local uniform limit of a sequence of functions. The idea is that every brunch of diameter $\delta>0$ in the set $T$, constructed in Subsection \ref{sec:Tree}, will correspond to a rotation and/or translation of the function $L_\delta$. Connecting subtrees will be done by taking maximum of some big constant times $L_{\delta^*}$ where $\delta^*$ is the diameter of the brunch connecting the subtrees in $T$.
\paragraph{A function corresponding to an outer subtree of rank $(k+1)$} We first point out that composing $L_\eps$ on a translation and/or a rotation does not change its growth-rate, only instead of taking supremum over $\abs{x_1}$ we should measure the distance along the translated (and/or rotated) $x_1$-axis from where the origin was translated. We will construct the function corresponding to the outer subtree, which is centered at $2^kv_0$, in $k$ steps.\\
For every $1\le j\le 2^d$ we keep the notation used in the previous subsection for vertices of the cube $[-1,1]^d$, $v_j$, and let $L_j$ be the function $L_{2^k\eps_k}$ translated and rotated so that the $x_1$-axis is aligned with the line connecting $2^kv_0+2^{k-1}v_j$ with $2^kv_0$, and the point $x=\bb{\frac{2\eps_k2^k\cdot d\log(2)}{\pi\sqrt{d-1}},0\cdots,0}$, is translated to the point $2^kv_0+2^{k-1}v_j$, while along the line described by the parametrisation $(1-t)\cdot \bb{2^kv_0+2^{k-1}v_j}+t\cdot2^kv_0$ the function $L_j$ increases with $t$. Let $R(x)$ be the rotation map taking the line connecting $2^kv_0$ and its `parent vertex', $2^{k+1}v_0$, to the $x_1$-axis, and define $L_*(x)=L_{2^k\eps_k}(R(x))$. Define the function
$$
v_1(x)=	\begin{cases}
		\max\bset{L_*(x), p_1\cdot L_j(x)}&, L_j(x)>0 \text{ and } x\nin R\inv\bb{G_{2^k\eps_k}}\\
		L_*(x)&, \text{otherwise}
		\end{cases},
$$
where $G_{2^k\eps_k}$ is the set defined in  property 2 of the function $L_\eps$. We will show that if $p_1$ is chosen correctly, then $v_1$ is well defined and subharmonic. It is enough to show that $p_1$ can be chosen so that on $\partial R\inv\bb{G_{2^k\eps_k}}$ already $L_*(x)>p_1\cdot L_j(x)$ for all $1\le j\le 2^d$. Using properties \ref{L_eps:low_bnd} and \ref{L_eps:up_bnd} of the function $L_\eps$ we see that
$$
\underset{x\in\partial R\inv\bb{G_{2^k\eps_k}}}\inf\; L_*(x)=\underset{x\in\partial G_{2^k\eps_k}}\inf\; L_{2^k\eps_k}(x)>1=p_1\cdot\underset{x\in\partial R\inv\bb{G_{2^k\eps_k}}}\sup\; L_j(x),
$$
if we choose 
$$
\frac1{p_1}:=\exp\bb{\frac{\pi d}{\eps_k}}\ge\underset{x\in\partial R\inv\bb{G_{2^k\eps_k}}}\sup\; L_j(x),
$$
since the distance between $2^kv_0+2^{k-1}v_j$ and $2^kv_0$ is bounded by $\sqrt d\cdot 2^{k-1}$. We continue to define the functions $v_2,\cdots,v_{s_k}$ in a similar manner with $L_j$ being translations and/or rotations of the function $L_{2^{k-m}\eps_k}$ and a sequence of constants $p_m$ satisfying that
$$
\frac{p_m}{p_{m+1}}=\exp\bb{\frac{\pi d}{\eps_k}}.
$$
We point out that on step $m$ we have $2^{d(k-m)}$ independent locations where we perform this `glueing procedure', but it is the same in every location (up to rotations and/or translations). We are therefore able to use the same constants, $\bset{p_m}$. After $s_k$ steps, we use the same construction but now with $L_{\eps_1}$ and $G_{\eps_1}$ instead of $L_{2^{k-m}\eps_k}$ and $G_{2^{k-m}\eps_k}$. Then for all $j>s_k$ we choose
$$
\frac{p_j}{p_{j+1}}=\exp\bb{\frac{\pi d\cdot2^j}{\eps_1}}.
$$
We then define $\tau_{k+1}(x):=\frac1{p_k}v_k(x)$.

For every basic cube, $I$, in which the tree $T$ is sparse, $M_{\tau_{k+1}}(I)\ge 1$ since along the centerline of every tube, $\tau_{k+1}\ge 1$. This makes every basic cube which does NOT intersect a brunch, a cube which satisfies both Property (\ref{eq:max}) and Property (\ref{eq:zero}), i.e $\tau_{k+1}$ oscillates in every basic cube which does not intersect a brunch.

To bound the growth of the function $\tau_{k+1}$ we use property \ref{L_eps:up_bnd} of the function $L_\eps$ for $\eps=2^k\eps_k=2^{s_k}$. Let $Q_k:=2^kv_0+[-2^k,2^k)^d$, then
\begin{eqnarray*}
\underset{x\in Q_k}\sup \tau_{k+1}(x)&=&\frac1{p_k}\underset{x\in Q_k}\sup L_*\le \frac1{p_1}\cdot\bb{\prodit j 1 {k-1}\frac{p_j}{p_{j+1}}}\cdot \exp\bb{\frac{\pi\sqrt{d-1}}{2^k\eps_k}\cdot \sqrt d\cdot 2^k}\\
&\le&\exp\bb{\pi d\bb{\frac{s_k}{\eps_k}+\frac1{\eps_1}\sumit j 1 {k-s_k}2^j}}\le\exp\bb{\pi d\bb{\frac{s_k}{\eps_k}+\frac1{\eps_1} 2^{k-s_k}}}\\
&=&\exp\bb{\frac{\pi d}{\eps_k}\bb{s_k+\frac1{\eps_1}}}\le\exp\bb{\frac{2\pi d\cdot s_k}{\eps_k}},
\end{eqnarray*}
for all $k$ large enough, by the maximum principle. Now, $s_k$ was chosen so that $1\le \frac{s_k^{\frac1{d-1}}2^{s_k}}{\bb{\frac{f(2^k)}{2^k}}^{\frac1{d-1}}}\le 4$, implying that
\begin{eqnarray*}
\underset{x\in Q_k}\sup \tau_{k+1}(x)&\le&\cdots\le \exp\bb{\frac{2\pi d\cdot s_k}{\eps_k}}=\exp\bb{\frac{2\pi d\cdot s_k\cdot 2^k}{2^{s_k}}}\le\exp\bb{\frac{2\pi d\cdot s_k\cdot 2^k}{\bb{\frac{f(2^k)}{s_k\cdot 2^k}}^{\frac1{d-1}}}}\\
&=&\exp\bb{\frac{2\pi d\cdot s_k^{\frac d{d-1}}\cdot \bb{2^{k}}^{\frac d{d-1}}}{f(2^k)^{\frac1{d-1}}}}\le \exp\bb{4\pi d\cdot \frac{2^{k\cdot\frac d{d-1}}}{f(2^k)^{\frac1{d-1}}}\log^{\frac d{d-1}}\bb{\frac{f(2^k)}{2^k}}}:=\mathcal M_k.
\end{eqnarray*}
Lastly, we point out that the support of the function $\tau_{k+1}$ is almost the same as an outer subtree of rank $k$ union with the support of the `handle', $L_*$. The difference is, we have additional small tips at the beginning of every tube, but the number of additional rogue basic cubes in these tips is bounded by their number in the entire brunch. Using an argument similar to the one done in Claim \ref{clm:bad_guys},
$$
\#\bset{\text{Rogue basic cubes in }[0,2^{k+1})^d}\le c_\tau f(2^{k+1}),
$$
for some uniform constant $c_\tau$.
\paragraph{Connecting outer subtrees of rank $k$ to the subtree $T_k$}
We will construct  the desirable function inductively, as a local uniform limit of a sequence of functions $\bset{u_k}$ satisfying that
\begin{enumerate}
\item $u_k$ is subharmonic.
\item $u_k$ is supported on $[0,2^k)^d$ and a translation and rotation of the function $L_{2^k\delta_k}$ (like a handle sticking out), for $\delta_k:=\bb{\frac{f(2^k)}{2^{k\cdot d}}}^{\frac1{d-1}}$, as defined in the previous subsection.
\item For every $j<k$ we have $u_k|_{[0,2^j)^d}\equiv u_j$.
\item $\underset{[0,2^k)^d}\sup u_k\le\mathcal M_k:= \exp\bb{4\pi d\cdot \frac{2^{k\cdot\frac d{d-1}}}{f(2^k)^{\frac1{d-1}}}\log^{\frac d{d-1}}\bb{\frac{f(2^k)}{2^k}}}$.
\item There exists a uniform constant $\gamma$ so that $\gamma([0,2^k)^d)\le\gamma\cdot f(2^k)$.
\end{enumerate}
Let $\tau_1$ be the function constructed above for $k=1$, and define $u_1(x)=\tau_1(x)$. This function is subharmonic as $\tau_1$ is, and properties 2,3,4 and 5 hold by default. Assume that $u_k$ was defined so that it satisfies properties 1,2,3,4 and 5 and let us construct the function $u_{k+1}$. We note that for every $k$, the function $\tau_k$ is supported on the outer subtree of rank $k$ union with the support of the function $L_*$, the latter being exactly translations and rotations of the function $L_{2^k\eps_k}$. Informally, to define $u_{k+1}$ we will `wrap' the support of $u_k$ with $2^d-1$ translations and/or rotations of the support of the function $\tau_k$, which we denote by $\tau_k^j$ for $1\le j\le 2^d-1$. We will only use translations and/or rotations of $\tau_k$ which are not the opposite of the direction of the component, as if not to have overlaps between the support of $u_k$ and the support of $\tau_k^j$. We then connect these functions with $L_{2^{k+1}\delta_{k+1}}$ in the `right direction'.

Formally, without loss of generality assume that $v_{2^d}=(-1,-1,\cdots,-1)$. We let $\tau_k^j(x)=\tau_k\circ R_j\inv(x-2^{k+1}v_0)$, where $R_j$ was the linear mapping taking $v_0+[-1,1]^d$ to $v_j+[-1,1]^d$ so that $R_j(0)=0$ and $R_j(v_0)=v_j$. We would like to `glue' together the definition of the function $u_k$ with the functions $\tau_k^j$. Let $L_{k+1}$ denote the composition of $L_{2^{k+1}\delta_{k+1}}$ with a translation and/or rotation which maps the set $\bset{x=(x_1,0,\cdots, 0)\in\R^d,\; x_1\ge\frac{2\eps_k2^k\cdot d\log(2)}{\pi\sqrt{d-1}}}$ to the line connecting $2^kv_0$ with $2^{k+1}v_0$. We will denote by $B_{k+1}$ the set $\R^d\setminus G_{2^{k+1}\delta_{k+1}}$ translated and rotated in the same manner as $L_{k+1}$. Let
$$
u_{k+1}(x)=	\begin{cases}
			\max\bset{\tau_k^j(x), \mathcal M_k\cdot L_{k+1}(x)}&, x\in B_{k+1}\\
			\mathcal M_k\cdot L_{k+1}(x)&, \text{otherwise}.
			\end{cases}
$$
To conclude that $u_{k+1}$ is subharmonic, it is enough show that on $\partial B_{k+1}$ already $\tau_k^j(x)<\mathcal M_k\cdot L_{k+1}(x)$ and so as subharmonicity is a local property, this implies that $u_{k+1}$ is subharmonic. A computation, similar to the one done in the construction of a function corresponding to an outer subtree of rank $k$, shows that
\begin{eqnarray*}
\underset{x\in\partial B_{k+1}}\inf\; \mathcal M_k\cdot L_{k+1}(x)\ge\frac{\mathcal M_k}{2^{d-1}}\cdot\cosh\bb{\frac{\pi\sqrt{d-1}}{2^{k+1}\cdot\delta_{k+1}}\cdot \frac{2^{k+1}\delta_{k+1}\cdot d\log(2)}{\pi\sqrt{d-1}}}&\ge&\mathcal M_k\cdot2^{-d}\exp\bb{\log\bb{2^d}}=\mathcal M_k,\\
\mathcal M_k&\ge&\underset{x\in\partial B_{k+1}}\sup\; \max\bset{\tau_k^1,\cdots,\tau_k^{2^d-1},u_k},
\end{eqnarray*}
based on the induction assumption and the bounds proved for the function $\tau_k$ in the previous paragraph. To conclude the proof we need to show that properties 4 and 5 hold for $u_{k+1}$, since the rest of the properties are obvious by the construction. 

To see property 5 holds, we note that if $c_\tau$ is a constant, so that the number of rogue basic cubes in the outer trees of rank $k$ is bounded by $c_\tau f(2^k)$, then by the induction assumption
\begin{eqnarray*}
\gamma_f^{u_{k+1}}\bb{\left[0,2^{k+1}\right)^d}&=&\gamma_f^{u_k}\bb{\left[0,2^k\right)^d}+2^{k+1}\bb{2^{k+1}\delta_{k+1}}^{d-1}+\bb{2^d-1}\gamma_f^{\tau_k}\bb{\left[0,2^k\right)^d}\\
&\le&\gamma\cdot f(2^k)+f\bb{2^{k+1}}+c_\tau\bb{2^d-1}f(2^k)=f\bb{2^{k+1}}\bb{1+\frac{f(2^k)}{f\bb{2^{k+1}}}\bb{\gamma+c_\tau\bb{2^d-1}}}\\
&\le&f\bb{2^{k+1}}\bb{1+\frac23\bb{\gamma+c_\tau\bb{2^d-1}}},
\end{eqnarray*}
as we saw that for $k$ large enough $\frac{f(2^k)}{f\bb{2^{k+1}}}\le\frac23$. If we choose $\gamma$ so that
$$
\gamma\ge1+\frac23\bb{\gamma+c_\tau\bb{2^d-1}}\iff \gamma\ge 3\bb{1+\frac23c_\tau\bb{2^d-1}},
$$
then 
\begin{eqnarray*}
\gamma_f^{u_{k+1}}\bb{\left[0,2^{k+1}\right)^d}&\le&f\bb{2^{k+1}}\bb{1+\frac23\bb{\gamma+c_\tau\bb{2^d-1}}}\le\gamma\cdot f\bb{2^{k+1}}
\end{eqnarray*}
and property 5 holds.

It is left to prove the upper bound on $u_{k+1}$, i.e to show that property 4 holds. Following the maximum principle and using property \ref{L_eps:up_bnd} of the function $L_{k+1}$ we see that
\begin{eqnarray*}
\underset{[0,2^{k+1})^d}\sup u_{k+1}=\mathcal M_k\cdot \underset{[0,2^{k+1})^d}\sup\;  L_{k+1}(x)&\le& \mathcal M_k\cdot \exp\bb{\frac{\pi\sqrt{d-1}}{2^{k+1}\delta_{k+1}}\cdot 2^{k+1}}\le\mathcal M_k\exp\bb{\frac{\pi d}{\delta_{k+1}}}.
\end{eqnarray*}
To conclude the proof of the upper bound, we need to show that the latter is bounded by $\mathcal M_{k+1}$, i.e
\begin{equation*}
\tag{$\clubsuit$} \mathcal M_k\cdot\exp\bb{\frac{\pi d}{\delta_{k+1}}}\le \mathcal M_{k+1} \label{eq:outPaste}.
\end{equation*}

To simplify of the notation, we let $\varphi(t):=\log^{\frac d{d-1}}\bb{\frac{f(t)}t}$. Then 
$$
\mathcal M_k:= \exp\bb{4\pi d\cdot \frac{2^{k\cdot\frac d{d-1}}}{f(2^k)^{\frac1{d-1}}}\log^{\frac d{d-1}}\bb{\frac{f(2^k)}{2^k}}}=\exp\bb{4\pi d\cdot \frac{2^{k\cdot\frac d{d-1}}}{f(2^k)^{\frac1{d-1}}}\varphi(2^k)},
$$
and therfore
\begin{eqnarray*}
\mathcal M_k\cdot \exp\bb{\frac{\pi d}{\delta_{k+1}}}\le\mathcal M_{k+1}&\iff& 4\pi d\cdot \frac{2^{k\cdot\frac d{d-1}}}{f(2^k)^{\frac1{d-1}}}\varphi(2^k)+\pi d\frac{2^{(k+1)\cdot\frac d{d-1}}}{f(2^{k+1})^{\frac1{d-1}}}\le 4\pi d\cdot \frac{2^{(k+1)\cdot\frac d{d-1}}}{f(2^{k+1})^{\frac1{d-1}}}\varphi\bb{2^{k+1}}\\
&\iff&\bb{\frac{f(2^{k+1})}{2^df(2^k)}}^{\frac1{d-1}}\cdot\varphi(2^k)+\frac14\le\varphi\bb{2^{k+1}}.
\end{eqnarray*}
Now, $f$ is regularly varying and so $f(t)=t^ag(t)$, where $a\in[1,d]$ (as $f(t)\gg t$) and $g$ is slowly varying. In particular,
\begin{eqnarray*}
\frac{f(2^{k+1})}{2^df(2^k)}=\frac{2^{a(k+1)}\cdot g\bb{2^{k+1}}}{2^d\cdot 2^{a\cdot k}g(2^k)}=2^{a-d}\cdot\frac{g\bb{2\cdot 2^k}}{g\bb{2^k}}\le\begin{cases}
																									2^{a-d}\bb{1+o(1)}&, a<d\\
																									1&, a=d
																									\end{cases}
\end{eqnarray*}
as $g$ is slowly varying in general, and monotone decreasing in the case $a=d$. Then (\ref{eq:outPaste}) holds if
$$
C_a\varphi(2^k)+\frac14\le \varphi\bb{2^{k+1}},
$$
for the constant
$$
C_a:=\begin{cases}
	2^{\frac{a-d}{d-1}}\bb{1+o(1)}&, a<d\\
	1&, a=d
\end{cases}.
$$

To conclude the proof we will use the following inequality: for every $x,y\ge 0$ and every $\alpha\ge1$ we have 
$$
x^\alpha+y^\alpha\le\bb{x+y}^\alpha.
$$
To see this inequality holds, fix $y$ and let $\psi(t):=\bb{t+y}^\alpha-\bb{t^\alpha+y^\alpha}$. Then $\psi'(t)=\alpha\bb{\bb{t+y}^{\alpha-1}-t^{\alpha-1}}\ge0$ as long as $y\ge0$ and $\alpha\ge1$, which implies that $\psi$ is monotone non-decreasing. As $\psi(0)=0$, we conclude that $\psi(t)\ge 0$ for all $t\ge0$, and the inequality follows.

Using this inequality, we see that (\ref{eq:outPaste}) holds if
\begin{align*}
&C_a^{\frac{d-1}d}\log\bb{\frac{f(2^k)}{2^k}}+\frac1{4^{\frac{d-1}d}}\le C_a^{\frac{d-1}d}\log\bb{\frac{f(2^k)}{2^k}}+\frac12\le\log\bb{\frac{f\bb{2^{k+1}}}{2^{k+1}}}\\
&\iff \bb{C_a^{\frac{d-1}d}-1}\log\bb{\frac{f(2^k)}{2^k}}+\frac12\le\log\bb{\frac{f\bb{2^{k+1}}}{2^{k+1}}}-\log\bb{\frac{f(2^k)}{2^k}}=\log\bb{\frac{f\bb{2^{k+1}}}{2f(2^k)}}.\tag{$\clubsuit\clubsuit$}\label{eq:Paste}
\end{align*}
To show that (\ref{eq:Paste}) holds and conclude the proof, we will look into two different cases:

$\bullet$ If $a<d$, then 
$$
C_a^{\frac{d-1}d}-1=2^{\frac{a-d}{d}}\bb{1+o(1)}-1\le 2^{\frac{a-d}{2d}}-1<0
$$
for $k$ large enough. On the other hand, since $f(t)\gg t$, then $\log\bb{\frac{f(2^k)}{2^k}}\rightarrow\infty$. We conclude that 
$$
\bb{C_a^{\frac{d-1}d}-1}\log\bb{\frac{f(2^k)}{2^k}}+\frac12\le \bb{2^{\frac{a-d}{2d}}-1}\log\bb{\frac{f(2^k)}{2^k}}+\frac12\rightarrow-\infty \text{ as }k\rightarrow\infty.
$$
In particular, the left hand side of (\ref{eq:Paste}) is bounded from above by $(-1)$ for all $k$ large enough.

We bound the right hand side of (\ref{eq:Paste}) from bellow by $(-1)$ since $a\ge1$ and so
$$
\log\bb{\frac{f\bb{2^{k+1}}}{2f(2^k)}}=\log\bb{2^{a-1}}+\log\bb{\frac{g(2\cdot 2^k)}{g(2^k)}}\rightarrow \log\bb{2^{a-1}}\ge 0\;\text{ as }k\rightarrow\infty,
$$
showing that if $a<d$, then (\ref{eq:Paste}) holds.

$\bullet$ If $a=d$, then the left hand side of (\ref{eq:Paste}) is $\frac12$, and since $g$ is slowly varying
$$
\log\bb{\frac{f\bb{2^{k+1}}}{2f(2^k)}}=\log\bb{2^{d-1}}+\log\bb{\frac{g(2\cdot 2^k)}{g(2^k)}}= \log\bb{2^{d-1}}-o(1)\ge \frac12,
$$
for all $k$ large enough, showing that if $a=d$, then (\ref{eq:Paste}) holds as well.

\paragraph{Defining a function on $\R^d$} Let $u_0$ denote the local uniform limit of this sequence of functions. To define the function $u$ we let
$$
u(x):=\sumit j 0 {2^d-1} u_0\circ R_j\inv(x),
$$
for the linear mappings $R_j$ taking $v_0$ to $v_j$ and the origin to itself.

It is left to bound the number of rogue basic cubes in every large cube centered at the origin. For every $k$
$$
\gamma_f^{u}\bb{[-2^k,2^k]^d}=2^d \gamma_f^{u_0}\bb{[0,2^k]^d}=2^d \gamma_f^{u_{k+1}}\bb{[0,2^k]^d}=2^d\cdot \frac{\#\bset{\text{ Rogue basic cubes in }[0,2^k]^d}}{f(2^k)}\le 2^d\cdot \gamma.
$$
For a general cube $Q=[-N,N]^d$ there exists $k$ so that $2^k\le N<2^{k+1}$ and as long as $N$ is large enough
$$
\gamma(Q)\le\frac{f(2^{k+2})\gamma([-2^{k+1},2^{k+1}]^d)}{f(2^{k+1})}\lesssim\frac{f(2^{k+2})}{f(2^{k+1})}=\frac{2^{\alpha\bb{k+2}}g\bb{2\cdot2^{k+1}}}{2^{\alpha\bb{k+1}}g\bb{2^{k+1}}}\le 2^{d+1}.
$$
Doing the same construction with $f$ rescaled so that the latter will be less than $1$, will only effect the growth by changing the constants, and so on one hand it will create an $f$-oscillating subharmonic function, on the other hand it will have the optimal growth as needed.

We conclude this section with two remarks:
\begin{rmk}
 One may wonder what happens if we require some oscillation, without bounding its magnitude from bellow, i.e replace property (\ref{eq:max}) with the property
 $$
 (\widetilde{P_1})\;\;\lambda_{d-1}(\bset{u>0}\cap I)>0.
 $$
 Using the same construction described above, with $\tau_k=v_k$, i.e without rescaling $v_k$ by $\frac1{p_k}$, generates an example where all but at most $f(N)$ basic squares in  $\left[-\frac N2,\frac N2\right]^d$ satisfy properties $(\widetilde{P_1})$ and $(\ref{eq:zero})$, while the growth of the function is bounded by a polynomial. This shows that a bound from bellow on the 
magnitude of the oscillation in each basic cube is essential.
\end{rmk}
\begin{rmk}
 We cannot use this construction to generate a translation invariant probability measure on the space of entire functions using the means described in \cite{Us2017}. The reason is, 
our function does not satisfy condition (10) in Lemma 7, that there exist a sequence of squares $\bset{S_k}\nearrow\C$ and a sequence of constants $\bset{\mathcal M_k}$ so that
$$
\limit k\infty\limitinf n\infty \frac{m_2\bb{\bset{w\in S_n,\; \underset{z\in w+S_k}\max\; u\le \mathcal M_k}}}{m_2(S_n)}=1.
$$
\end{rmk}

\section{Appendices}
\subsection{Appendix A- A proof for Claim \ref{clm:HarmonicBnd}}
The proof we present here is a concatenation of two inequalities. To state and prove these inequalities we will need the following definitions: Define the function
$$
k_d(t):=	\begin{cases}	
		\log \bb{t}&, d=2\\
		\frac{-1}{t^{d-2}}&,d\ge 3
		\end{cases}.
$$
Following \cite{Bishoperes}, for every measure $\nu$ we let
\begin{eqnarray*}
p_\nu(x)&:=&\integrate{\R^d}{}{k_d\bb{\abs{x-y}}}\nu(y)\\
I(\nu)&:=&\integrate{\R^d}{}{p_\nu(x)}\nu(x)=\integrate{\R^d}{}{\integrate{\R^d}{}{k_d\bb{\abs{x-y}}}\nu(y)}\nu(x).
\end{eqnarray*}
The first is called {\it the potential of the measure $\nu$} and the second is called {\it the energy of the measure $\nu$}. It is known that if the set $E$ is compact, then there exists a probability measure $\nu_0$ so that
$$
I(\nu_0)=\underset{\nu}\sup\; I(\nu)
$$
where the supremum is takes over all the Borel probability measures $\nu$ which are supported on $E$. The measure $\nu_0$ is called {\it the equilibrium measure of $E$}. For more information see, for example, Theorem 5.4 on page 209 of \cite{Haywood}.

We are now ready to state these inequalities:
\begin{claim}\label{clm:Hausdorff}
Let $E\subset B\bb{0,\frac12}$ be a compact set with $\lambda_{d-1}(E)>0$. If $\nu_0$ is the equilibrium measure of $E$, then
$$
\lambda_{d-1}(E)\lesssim_d\; \frac{-1}{I(\nu_0)}.
$$
\end{claim}
The proof of this claim is a variation of the proof on p.94 in \cite{Bishoperes}. For the sake of completeness, and since the statement is not exactly the same, we bring here the full proof.

We will use the following result:
\begin{lem}[Frostman's Lemma]\label{lem:Frostman}
Let $\varphi$ be a gauge function, i.e a positive, increasing function on $[0,\infty)$ with $\varphi(0) = 0$. Let $K\subset\R^d$ be a compact set with positive $\varphi$-Hausdorff content, $\lambda_\varphi(K) > 0$. Then there is a positive Borel measure $\mu$ on $K$ satisfying that for every ball of radius $r$, $B$, $\mu(B)\le C_d\cdot\varphi(r)$ while $\mu(K)\ge\lambda_\varphi(K)$. The constant $C_d$ depends on the dimension alone.
\end{lem}
This lemma, and its proof can be found for example as Lemma 3.1.1, p.83 in \cite{Bishoperes}.
\begin{proof}[Proof of Claim \ref{clm:Hausdorff}]
Let $\mu$ be the measure we obtain by using Lemma \ref{lem:Frostman} with $\varphi(t)=t^{d-1}$ and $K=E$. We begin by bounding from bellow the potential of the measure $\mu$. Then since $-k_d$ is a monotone decreasing positive function on $[0,1]$ and $diam(E)\le 1$
\begin{eqnarray*}
-p_\mu(x)&:=&\integrate{E}{}{-k_d(\abs{x-y})}\mu(y)\le\sumit n 0\infty \integrate{\abs{x-y}<2^{-n-1}}{\abs{x-y}<2^{-n}}{-k_d(\abs{x-y})}\mu(y)\\
&\le&\sumit n 0\infty \left\{-k_d\bb{2^{-n-1}}\left[\mu\bb{\bset{\abs{x-y}<2^{-n}}}-\mu\bb{\bset{\abs{x-y}<2^{-n-1}}}\right]\right\}\\
&\le&C_d\sumit n 0 \infty 2^{-n(d-1)} \bb{-k_d\bb{2^{-(n+1)}}}=2^{d-1}\cdot C_d\sumit n 0 \infty 2^{-(n+1)(d-1)}\bb{ -k_d\bb{2^{-(n+1)}}}\\
&\le&2^{d-1}\cdot C_d\sumit n 1 \infty 2^{-n(d-1)}\bb{ -k_d\bb{2^{-n}}}=	\begin{cases}
														2^{d-1}\cdot C_d\sumit n 1 \infty \log\bb{2^n}2^{-n}& , d=2\\
														2^{d-1}\cdot C_d\sumit n 1 \infty \frac{2^{n(d-2)}}{2^{n(d-1)}}&, d\ge 3
														\end{cases}
\end{eqnarray*}
both series are convergent and bounded by 2. We conclude that $-p_\mu(x)\le 2^dC_d$. Next, define the measure $d\nu=\frac{d\mu}{\mu(E)}$. Then $\nu$ is a probability measure supported on $E$, while
$$
I(\nu)=\frac1{\mu(E)^2}I(\mu)=\frac{-1}{\mu(E)^2}\integrate{\R^d}{}{-p_\mu(x)}\mu(x)\ge\frac{-2^dC_d}{\mu(E)}\ge\frac{-2^dC_d}{\lambda_{d-1}(E)}.
$$
By definition of equilibrium measure we see that
$$
I(\nu_0)=\underset{\mu}\sup\; I(\mu)\ge I(\nu)\ge \frac{-2^dC_d}{\lambda_{d-1}(E)}\Rightarrow \lambda_{d-1}(E)\le\frac{-2^dC_d}{I(\nu_0)},
$$
as $I(\nu_0)<0$, concluding our proof.
\end{proof}

\begin{claim}\label{clm:harm_kib}
Let $E\subset B\bb{0,\frac12}$ be a compact set with $\abs{I(\nu_0)}<\infty$. Then
$$
\omega(0,E;B(0,1)\setminus E)\gtrsim_d\; \frac{-1}{I(\nu_0)}.
$$
\end{claim}
\begin{proof}
Define the function $w:B(0,1)\rightarrow\R$ by
\begin{equation*}
w(x)=	\begin{cases}
			\omega\bb{x,E,B\bb{0,1}\setminus E} & \xi\in B\bb{0,1}\setminus E\\
			1 & \text{otherwise}
		\end{cases}.
\end{equation*}
This function is super-harmonic, and by Poisson-Jensen formula for every $x\in B\bb{0,1}$, since $w|_{\partial B(0,1)}\equiv 0$,
\begin{equation*}
\tag{$\star$}-w(x)= -\integrate{B\bb{0,1}}{}{G_B\bb{y,x}}\mu_w(y),\label{eq:w}
\end{equation*}
where $G_B$ is Green's function on $B=B\bb{0,1}$, and $\mu_w$ the Reisz measure of $w$, defined by $d\mu_w=\frac{\Delta w}{\sigma(\partial B(0,1))}dm_d$, where $\sigma_d$ is the surface measure in $\R^d$. As $E\subseteq B\bb{0,\frac12}$, there exists a constant $c_d>0$ so that $G_B(0,y)>c_d$ for every $y\in E$, implying that 
\begin{equation*}
\tag{$\star\star$}\omega\bb{0,E,B\bb{0,1}\setminus E}=w\bb{0}\overset{\text{by (\ref{eq:w})}}=\integrate E{}{G_B\bb{0,y}}\mu_w\bb{y}\ge c_d\cdot\integrate E{}{}\mu_w\bb{y}=c_d\cdot\norm{\mu_w}{}.\label{eq:wBnd}
\end{equation*}
Next, let $\nu$ denote the equilibrium measure of the set $E$, which is compact. For every measure $\mu$ supported on $E$, by Frostman's theorem:
\begin{equation*}
\norm\mu{}=\mu(E)=\frac{-1}{I(\nu)}\integrate{\R^d}{}{\integrate{\R^d}{}{-k_d(\abs{x-y})}\nu(x)}\mu(x)\gtrsim_d \frac{-1}{I(\nu)}\integrate E{}{\integrate E{}{G_B\bb{x,y}}\mu(x)}\nu(y),
\end{equation*}
for $G_B$ Green's function on $B=B\bb{0,1}$. \\
In particular, for $\mu=\mu_w$ we get that
$$
\norm{\mu_w}{}\ge\frac{-1}{I(\nu)}\integrate E{}{\integrate E{}{G_B\bb{x,y}}\mu_w(y)}\nu(x)\overset{\text{by (\ref{eq:w})}}=\frac{-1}{I(\nu)}\integrate E{}{w\bb{x}}\nu(x)\overset{(a)}=\frac{-1}{I(\nu)}\integrate E{}{}\nu(y)\overset{(b)}=\frac{-1}{I(\nu)},
$$
where $(a)$ is since $w|_E\equiv1$, and $(b)$ is since $\nu$ is a probability measure supported on $E$.\\
Combining this with (\ref{eq:wBnd}), we see that
$$
\omega\bb{0,E,B\bb{0,1}\setminus E}\gtrsim_d\;\norm{\mu_w}{}\ge \frac{-1}{I(\nu)},
$$
concluding the proof.
\end{proof}
Combining these two claims we see that
$$
\omega(0,E; B(0,1)\setminus E)\ge\frac{-C_1}{I(\nu)}\ge C_1\cdot C_2\cdot\lambda_{d-1}(E),
$$
concluding the proof of Claim \ref{clm:HarmonicBnd}.

\subsection{Appendix B- Another proof of the Main Lemma}\label{subsec:sodin}
It is possible to prove the Main Lemma, Lemma \ref{lem:Jones-Makarov}, by using the Main Lemma in \cite{JoneMakarov1995} with out repeating Jones and Makarov's ingenious idea. This is done using the following lemma, suggested to the author by M.Sodin:
\begin{lem}\label{lem:sodin}
Let $u$ be a subharmonic function defined in a neighbourhood of $Q_0:=\left[- N, N\right]^d$ for some $N\gg 1$ and define $Q=\left[-\frac N{4\sqrt d},\frac N{4\sqrt d}\right]^d$. Let $E\subseteq Z_u \cap Q$ be a closed set. Then for every basic cube $I\subset Q$ satisfying that $\lambda_{d-1}(E\cap I)>\eps_d>0$
$$
\frac{M_u(I)}{M_u(Q_0)}\lesssim\frac{\omega(I^*)}{\omega(E)},
$$
where $\omega(\cdot)=\omega(\infty,\;\cdot\;;\R^d\setminus E)$ and $I^*$ is a cube concentric with $I$ having edge length which depends on the dimension alone.
\end{lem}
In light of this lemma, and using the Main Lemma in \cite{JoneMakarov1995}, we get another proof for Lemma \ref{lem:Jones-Makarov}. For completeness, we bring here a reformulation of the Main Lemma in \cite{JoneMakarov1995}:
\begin{lem}[The Main Lemma in \cite{JoneMakarov1995}]
Let $E\subset B(0,1)$ be a compact set , and let $\omega(\cdot)=\omega(\infty,\cdot;\R^d\setminus E)$. We subdivide $[0,1]^d$ into $N^d$ cubes I with side length $\frac1N$ and denote by $\mathcal G$ the whole collection of cubes. We define a subset $\mathcal E\subset\mathcal G$ ("empty" squares), as follows:
$$
I\in\mathcal E \text{ if }\lambda_{d-1}(E\cap I)\le\frac{\eps_d}{N^{d-1}},
$$
where $\lambda_{d-1}$ is the $(d-1)$-dimensional dyadic Hausdorff content. If $\#\mathcal E\le c_d\cdot N^d$ for some absolute constant $c_d$, then for at least half of the cubes $I\in\mathcal G$ the following inequality holds:
$$
\frac{\omega(I)}{\omega([0,1]^d)}<\exp \bb{-\alpha_d\bb{\frac {N^d}{N+\#\mathcal E}\log^d\bb{2+\frac{\#\mathcal E}N}}^{\frac1{d-1}}},
$$
where $\alpha_d$ is an absolute constant.
\end{lem}

\begin{proof}[Alternative proof of Lemma \ref{lem:Jones-Makarov}]
Following Lemma \ref{lem:sodin}, it is enough to bound $\omega(I^*)$ from above.\\
Let $I\in\mathcal G\setminus\mathcal E$ be a basic cube and let $I^*$ be the corresponding  cube from Lemma \ref{lem:sodin}. We rescale the set $E$ by $\frac1{\ell(I^*)N}$ so that rescaling $I^*$, we get a cube of edge-length $\frac1N$, and $E\subset B(0,1)$. We may now apply the Main Lemma in \cite{JoneMakarov1995}, on the rescaled set, to conclude the proof.
\end{proof}

\subsection{The proof of Lemma \ref{lem:sodin}}
\begin{proof}[\nopunct]
Let $g$ denote Green's function for the domain $\Omega:=\R^d\setminus E$ with pole at $\infty$. Since $E\subseteq Q\subset B\bb{0,\frac N2}$ and Green's functions satisfy the subordination principle, for every $\abs x\ge N$
$$
g(x)\ge g_{\R^d\setminus B\bb{0,\frac N2}}(\infty,x)\ge\underset{\abs x=N}\min\; g_{\R^d\setminus B\bb{0,\frac N2}}(\infty,x)=\underset{\abs x=2}\min\; g_{\R^d\setminus B(0,1)}(\infty,x):=c_1,
$$
where $c_1$ is a constant which depends on the dimension, but independent of $N$. By the maximum principle, for every subharmonic function, $u$, with $E\subset Z_u$,
$$
\frac{u(x)}{M_u(Q_0)}\le 1\le\frac{g(x)}{c_1},
$$
implying that
$$
\frac{M_u(I)}{M_u(Q_0)}\lesssim M_g(I).
$$
To conclude the proof we shall bound $M_g(I)$ from above.

Let $h$ be a subharmonic function and let $I$ be a basic cube satisfying property (\ref{eq:zero}), and let $x_I\in\partial I$ be a point satisfying that $M_h(I)=h(x_I)$. By taking a square double in edge length, we may assume without loss of generality that $dist(x_I,E\cap I)>\frac13$. We note that following Nevanlinna's first fundamental theorem (see for example \cite[Theorem 3.19]{Haywood}), for every $r>\sqrt d$
$$
M_h\bb{B\bb{x_I,\sqrt d}}\le \frac{r^{d-2}(r+\sqrt d)}{(r-\sqrt d)^{d-1}}\bb{h(x_I)+\underset{y\in E\cap I}\max G_{B(x_I,r)}(x_I,y)\cdot\mu_h(B(x_I,r))},
$$
where $G_{B(x_I,r)}$ is Green's function for a ball of radius $r$ centered at $x_I$, and $\mu_h$ is Riesz measure, defined by $d\mu_h(x)=\frac{\Delta h(x)}{\sigma_d(\partial B(0,1))}dm_d(x)$, for $\sigma_d$ the surface area measure in $\R^d$.\\
If $r=\alpha\cdot\sqrt d$, then for every $d$ fixed
$$
\frac{r^{d-2}(r+\sqrt d)}{(r-\sqrt d)^{d-1}}=\frac{\alpha^{d-2}d^{\frac{d-2}2}\cdot\sqrt d(\alpha+1)}{d^{\frac{d-1}2}(\alpha-1)^{d-1}}=\frac{\alpha^{d-2}(\alpha+1)}{(\alpha-1)^{d-1}}=\frac{1+\frac1\alpha}{\bb{1-\frac1\alpha}^{d-1}}\searrow 1, \text{ as } \alpha\rightarrow\infty.
$$
We conclude that for every $\delta>0$ there exists $\alpha$ so that
\begin{equation*}
\tag{$\star$} M_h\bb{B\bb{x_I,\sqrt d}}\le \bb{1+\delta}\bb{h(x_I)+\underset{y\in E\cap I}\max G_{B(x_I,\alpha\sqrt d)}(x_I,y)\cdot\mu_h(B(x_I,\alpha\sqrt d))}. \label{eq:Nevanlinna}
\end{equation*}
Next, since the basic cube $I$ satisfies that $\lambda_{d-1}(E\cap I)>\eps_d>0$, following Observation \ref{obs:jump}, there exists $\delta_d>0$ for which
$$
h(x_I)<M_h\bb{B\bb{x_I,\sqrt d}}(1-\delta_d).
$$
Combining this with (\ref{eq:Nevanlinna}), we note that if we choose $\alpha>0$ large enough (depending on the dimension $d$, and, in particular, on $\delta_d$)
\begin{eqnarray*}
M_h\bb{B\bb{x_I,\sqrt d}}&\le& \bb{1+\frac{\delta_d}2}\bb{h(x_I)+\underset{y\in E\cap I}\max G_{B(x_I,\alpha\sqrt d)}(x_I,y)\cdot\mu_h(B(x_I,\alpha\cdot\sqrt d))}\\
&\le&\bb{1+\frac{\delta_d}2}\bb{(1-\delta_d)M_h\bb{B\bb{x_I,\sqrt d}}+c_d\cdot\mu_h(B(x_I,\alpha\cdot\sqrt d))},
\end{eqnarray*}
for some constant $c_d$, which depends on the dimension alone.

This implies that if $I^*$ is a cube centred at $x_I$ with edge length $2\alpha\sqrt d$ then
$$
M_h(I)\le M_h\bb{B\bb{x_I,\sqrt d}}\le \frac{4c_d}{\delta_d}\cdot \mu_h(B(x_I,\alpha\cdot\sqrt d))\le\frac{4c_d}{\delta_d}\cdot \mu_h(I^*).
$$
In particular, for $h=g$ we see that
$$
M_g(I)\lesssim\mu_g(I^*).
$$
To conclude the proof, we use the fact that for every cube $J$:
$$
\mu_g(J)=\sigma_d\bb{B(0,1)}\cdot\omega(\infty,J;\Omega),
$$
for $\sigma_d$ the surface area in $\R^d$, implying that
$$
\frac{M_u(I)}{M_u(Q_0)}\lesssim M_g(I)\lesssim\mu_g(I^*)\sim_d\omega(\infty,I^*,\R^d\setminus E)\le\frac{\omega(I^*)}{\omega(E)}.
$$
\end{proof}


\nocite{*}
\bibliographystyle{plain}
\bibliography{}

\begin{thebibliography}{}

\bibitem{Baranov2008}
A.~Baranov and  H.~ Hedenmalm.
\newblock{Boundary properties of Green functions in the plane}.
\newblock{\em Duke Math. J.},145:1-24, 2008.

\bibitem{Bingang}
N.H ~Bingham, C.M~ Goldie, and J.L Teugels.
\newblock {Regular Variation}.
\newblock {\em Encyclopedia of Mathematics and its Applications, Cambridge University Press}, 1987.


\bibitem{Bishoperes}
 C. Bishop, and Y. Peres.
\newblock{Fractals in probability and analysis},
\newblock {\em Cambridge Studies in Advanced Mathematics,Vol 162, Cambridge University Press, Cambridge}, 2017.


\bibitem{Us2017}{L.~Buhovsky, A.~Gl\"ucksam, A.~Logunov, and M.~Sodin},
\newblock {Translation-invariant probability measures on entire functions}.
\newblock {\em Journal d'Analyse Mathématique, 1-33}, 2019.

\bibitem{Carleson}{L.Carleson},
\newblock {Selected Problems on Exceptional Sets},
\newblock {\em Van Nostrand}, 1967.

\bibitem{Haywood}
W. K. Hayman and P. B. Kennedy.
\newblock Subharmonic functions Vol. 1.
\newblock {\em London Mathematical Society Monographs No 9, Academic Press, London}, 1976.

\bibitem{JoneMakarov1995}
P.~W. Jones and N.G. Makarov.
\newblock Density properties of harmonic measure.
\newblock {\em Ann. of Math. (2)}, 142(3):427--455, 1995.


\end{thebibliography}

\bigskip
\noindent A.G.:
Mathematical and Computational Sciences, University of Toronto Mississauga, Canada.
\newline{\tt adiglucksam@gmail.com} 

\end{document}